\documentclass[letterpaper,11pt]{amsart}

\usepackage{amsmath,amssymb,amsxtra,amsthm,hyperref}
\usepackage{tikz}

\usepackage{chngcntr}



\newcommand{\bh}[1] {\mathcal{B}(\mathcal{#1})}

\newcommand{\lspan} {\operatorname{span}}
\newcommand{\lcm} {\operatorname{lcm}}

\newcommand{\cA}{\mathcal{A}}
\newcommand{\cB}{\mathcal{B}}
\newcommand{\cD}{\mathcal{D}}

\newcommand{\cF}{\mathcal{F}}

\newcommand{\cH}{\mathcal{H}}

\newcommand{\cJ}{\mathcal{J}}
\newcommand{\cK}{\mathcal{K}}
\newcommand{\cL}{\mathcal{L}}
\newcommand{\cM}{\mathcal{M}}

\newcommand{\cO}{\mathcal{O}}

\newcommand{\cQ}{\mathcal{Q}}

\newcommand{\cS}{\mathcal{S}}
\newcommand{\cT}{\mathcal{T}}

\newcommand{\fA}{\mathfrak{A}}

\newcommand{\fT}{\mathfrak{T}}

\newtheorem{theorem}{Theorem}[section]

\newtheorem{corollary}[theorem]{Corollary}
\newtheorem{proposition}[theorem]{Proposition}    
\newtheorem{lemma}[theorem]{Lemma}

\theoremstyle{definition}
\newtheorem{definition}[theorem]{Definition}

\newtheorem{example}[theorem]{Example}
\newtheorem{remark}[theorem]{Remark}

\numberwithin{equation}{section}

\begin{document}

\title[$C^*$-Envelope of Semigroup Dynamical Systems]{$C^*$-Envelope and Dilation Theory of Semigroup Dynamical Systems}

\author{Boyu Li}
\address{Department of Mathematics and Statistics, University of Victoria, Victoria, B.C. V8W 3R4}
\email{boyuli@uvic.ca}
\date{\today}


\begin{abstract} In this paper, we construct, for a certain class of semigroup dynamical systems, two operator algebras that are universal with respect to their corresponding covariance conditions: one being self-adjoint, and another being non-self-adjoint. We prove that the $C^*$-envelope of the non-self-adjoint operator algebra is precisely the self-adjoint one. This result leads to a number of new examples of operator algebras and their $C^*$-envelopes, with many from number fields and commutative rings. We further establish the functoriality of these operator algebras along with their applications.
\end{abstract}

\maketitle

\section{Introduction}

The study of operator algebras associated with dynamical systems has a rich and profound history. The dynamical systems and actions are encoded by operators on Hilbert space, and at the same time, the operator algebra often reflects important properties of the dynamics. For example, a group dynamical system $(\cA,G,\alpha)$ can be encoded by a representation $\pi$ of the $C^*$-algebra $\cA$ and a unitary representation $U$ of $G$. The automorphic action $\alpha$ is then encoded by the covariance relation 
\[U(g)\pi(a) U(g)^* = \pi(\alpha_g(a)).\]
This leads to the construction of the $C^*$-algebra cross products. Properties of the action are often translated to properties of the operator algebra, and vice versa. 

Unlike group dynamical systems, a semigroup dynamical system $(\cA,P,\alpha)$ usually assumes an endomorphic action due to the lack of inverses in semigroups. Because the endomorphic action is no longer automorphic and can be non-unital, a semigroup dynamical system is often encoded by an isometric (instead of unitary) representation $V$ of the semigroup $P$ and  a representation $\pi$ of the $C^*$-algebra $\cA$, such that $(\pi,V)$ satisfies the covariance relation
\[V(p)\pi(a) V(p)^* = \pi(\alpha_p(a)).\]

Many well-known $C^*$-algebras naturally inherit a semigroup dynamical system structure. For example, the Cuntz algebra $\cO_k$ can be constructed by an $\mathbb{N}$-action on the UHF-algebra of type $k^\infty$ \cite{Cuntz1977}. The semigroup $C^*$-algebra associated with a quasi-lattice ordered semigroup, first studied by Nica  \cite{Nica1992}, also encodes a semigroup dynamical system structure, where the semigroup $P$ acts on the diagonal algebra $\cD_P$ by $\ast$-endomorphisms \cite{LacaRaeburn1996}.

Semigroup dynamical systems help generate examples not only of $C^*$-algebras but also of non-self-adjoint operator algebras. Peters  \cite{Peters1984} was the first to consider relaxing the covariance relation to accommodate the absence of the adjoint in a non-self-adjoint operator algebra. In particular, there are two possibilities that replace the covariance relation: either 
\[V(p)\pi(a) = \pi(\alpha_p(a)) V(p);\]
or,
\[\pi(a) V(p) =V(p) \pi(\alpha_p(a)).\]
Peters focused on the latter relation and defines a non-self-adjoint operator algebra that he called the semicrossed product. This notion has since been generalized, most notably to semigroup dynamical systems over abelian lattice ordered semigroups  \cite{DFK2014}. We, on the other hand, focus on the former covariance relation that occurs naturally by multiplying $V(p)$ on the right of the covariance condition $V(p)\pi(a) V(p)^* = \pi(\alpha_p(a))$.

Non-self-adjoint and self-adjoint operator algebras are connected by $C^*$-envelope: a non-self-adjoint operator algebra can always be embedded in a smallest $C^*$-algebra, known as its $C^*$-envelope. The existence of the $C^*$-envelope was conjectured by Arveson in  \cite{ArvesonSubalgI} and proved by Hamana \cite{Hamana1979} (see also \cite{MS1998, DK2015}). Since then, there has been much research devoted to the computation of $C^*$-envelopes for various non-self-adjoint operator algebras \cite{DK2011, DK2012, EvgElias2012, KP2013, DFK2014, Kakariadis2016, KR2019, AdamElias2020}. For example, for an abelian lattice ordered semigroup $P$ inside a group $G=P^{-1}P$, the $C^*$-envelope of the semicrossed product $\cA\rtimes_\alpha^\cF P$ is often related to the crossed product $\tilde{\cA}\rtimes_\beta G$ \cite{EvgElias2012, Fuller2013, DP2015}, where $\tilde{\cA}$ is a direct limit from the endomorphic actions of $\cA$, and $\beta$ is an automorphic $G$-action. One may refer to \cite{DFK2014} for a more detailed description (see also  \cite{Humeniuk2020}).

This paper is motivated by a study of the $C^*$-envelope of a specific semigroup dynamical system arising from number fields \cite{Jaspar2016}. Let $R$ be the algebraic integers in a number field $K$. Cuntz, Deninger, and Laca constructed a Toeplitz-type $C^*$-algebra $\mathfrak{T}_R$ \cite{CDL2013} (that turns out to be the semigroup $C^*$-algebra for the $ax+b$ semigroup $R\rtimes R^\times$, see also \cite{XLi2012}). One can also construct a non-self-adjoint operator algebra analogue through a semigroup dynamical system, denoted by $\cA_R \rtimes^{iso}_\alpha P$. It is shown in \cite{Jaspar2016} that the $C^*$-envelope of $\cA_R \rtimes^{iso}_\alpha P$ is precisely $\mathfrak{T}_R$. In this paper, we generalize this result to a large class of semigroup dynamical systems. 

We shall start by reviewing the basic setups of semigroup dynamical systems in Section \ref{sec.dym}. For each semigroup dynamical system $(\cA,P,\alpha)$, we associate two operator algebras in Section \ref{sec.rep}: a self-adjoint $C^*$-algebra $\cA\rtimes_\alpha^* P$, which is universal with respect to isometric covariant representations; and a non-self-adjoint operator algebra $\cA\rtimes_\alpha^{iso} P$, which is universal with respect to isometric right covariant representations. The main result (Theorem \ref{thm.envelope}) states that the $C^*$-algebra $\cA\rtimes_\alpha^* P$ is the $C^*$-envelope of $\cA\rtimes_\alpha^{iso} P$. Among several approaches to compute the $C^*$-envelope, Dritschel and McCullough \cite{DM2005} prove that the $C^*$-envelope is generated by maximal representations. Our proof follows this approach: we first show that isometric covariant representations are maximal representations of $\cA\rtimes^{iso}_\alpha P$ (Proposition \ref{prop.maximal}). We then prove the converse: every isometric right covariant representation that is not isometric covariant has a non-trivial dilation (Proposition \ref{prop.dilation.base}). Consequently, isometric covariant representations are precisely the maximal representations, and by the result of Dritschel and McCullough, they generate the $C^*$-envelope (Theorem \ref{thm.envelope}). In Section \ref{sec.functor}, we study the functoriality for both the non-self-adjoint  $\cA\rtimes^{iso}_\alpha P$ and the $C^*$-algebra  $\cA\rtimes^*_\alpha P$. We prove that the functoriality of the $C^*$-algebra by injective $*$-homomorphism is equivalent to the functoriality of the non-self-adjoint operator algebra by complete isometry (Corollary \ref{cor.functor}). 

Our result introduces a rich collection of examples of non-self-adjoint operator algebras and their $C^*$-envelopes. In Section \ref{sec.ex}, we explore the examples arising from the semigroup $R\rtimes M$, where $R$ is a commutative ring and $M$ is a left-cancellative sub-semigroup of $R^\times$. We introduce a semigroup dynamical system $(\cA_{R,M}, M, \alpha)$ and prove that the $C^*$-algebra $\cA_{R,M}\rtimes^*_\alpha M$ is precisely the semigroup $C^*$-algebra $C^*(R\rtimes M)$ in the sense of Xin Li (Proposition \ref{prop.iso1}). Our main result applies to show that the $C^*$-envelope of $\cA_{R,M}\rtimes_\alpha^{iso} M$ is the semigroup $C^*$-algebra $C^*(R\rtimes M)$.
Finally, in Section \ref{sec.misc}, we consider additional examples and applications of our main result. 

I would like to thank Kenneth Davidson, Adam Dor-on, Chris Bruce, Evgenios Kakariadis, and Nadia Larsen for their many helpful comments, and Marcelo Laca for directing me to consider the functoriality. 

\section{Semigroup Dynamical Systems}\label{sec.dym}

\begin{definition} A \emph{semigroup dynamical system} is a triple $(\cA, P, \alpha)$ where
\begin{enumerate}
\item $\cA$ is a unital $C^*$-algebra;
\item $P$ is a semigroup with an identity;
\item $\alpha$ is a $P$-action on $\cA$ by $\ast$-endomorphisms. 
\end{enumerate}
\end{definition} 

We use $1\in\cA$ to denote the identity element of $\cA$ and $e\in P$ to denote the identity element of $P$. Though not required by the definition, \textbf{we assume throughout the paper that $P$ is an abelian left-cancellative semigroup}. The endomorphism $\alpha_p$ is a $\ast$-homomorphism from $\cA$ to itself, and it is neither assumed to be automorphic nor unital. The dynamical system is called injective if $\alpha_p$ is an injective $\ast$-endomorphism for all $p\in P$.

Recall that a subalgebra $\cB\subset \cA$ is called \emph{hereditary} if for any positive $b\in \cB$ and any $a\in \cA$ with $0\leq a\leq b$, we have $a\in \cB$ as well. We assume further that $\alpha_p(\cA)$, the image of each $\alpha_p$, is hereditary. Injective semigroup dynamical systems where each $\alpha_p(\cA)$ is hereditary are well-studied. For example, the following characterization is taken from Murphy \cite[Lemma 4.1]{Murphy1996}.

\begin{proposition}\label{prop.hereditary} Let $(\cA,P,\alpha)$ be an injective semigroup dynamical system. Then the following are equivalent:
\begin{enumerate}
\item $\alpha_p(\cA)$ is hereditary.
\item $\alpha_p(\cA)=\alpha_p(1)\cA\alpha_p(1)$.
\end{enumerate}
\end{proposition}  

Finally, for each $p\in P$, $\alpha_p(1)$ satisfies $\alpha_p(1)=\alpha_p(1)^2=\alpha_p(1)^*$, so that it is an orthogonal projection. In many cases, $\{\alpha_p(1)\}$ is a family of commuting projections.

\begin{definition} Let $(\cA,P,\alpha)$ be a semigroup dynamical system. We say that \emph{its projections are aligned} if for any $p,q\in P$, $\alpha_p(1)$ and $\alpha_q(1)$ commute. 
\end{definition} 

Throughout this paper, unless stated otherwise, \textbf{we always assume the semigroup dynamical system $(\cA,P,\alpha)$ is injective, $\alpha_p(\cA)$ is hereditary for all $p\in P$, and its projections are aligned. }

\begin{example} Let $X$ be a compact Hausdorff space, and $\sigma$ be an injective $P$-action on $X$ so that $X_p=\sigma_p(X)$ is a clopen subset of $X$. This induces a $P$-action $\alpha$ on $C(X)$ by $\alpha_p(f)(x)=f(\sigma_p^{-1}(x))$, which is a $\ast$-endomorphism from $C(X)$ onto $C(X_p)$. As a result, we obtain an injective semigroup dynamical system $(C(X),P,\alpha)$ where each $\alpha_p(C(X))=C(X_p)$ is an ideal of $C(X)$ and thus hereditary. Each $\alpha_p(1)$ is the characteristic function for $X_p$, and together they form a commuting family of orthogonal projections. 

As a special case when $X$ is a singleton, $\cA=\mathbb{C}I$ and $\alpha_p=id$ is the identity maps for all $p\in P$. We obtain a trivial semigroup dynamical system $(\mathbb{C}, P, id)$. 
\end{example} 

\begin{example} Let $(G,P)$ be a quasi-lattice ordered group, and $C^*(P)$ be its universal semigroup $C^*$-algebra with respect to isometric Nica-covariant representation 
 \cite{Nica1992}. Let $V$ be the universal Nica-covariant representation. Define the diagonal subalgebra  $\cD_P$ to be 
 \[\cD_P=\overline{\lspan}\{V(p)V(p)^*: p\in P\}.\]
Then $C^*(P)$ naturally encodes an injective semigroup dynamical system structure $(\cD_P,P,\alpha)$, where $\alpha_p(x)=V(p)x V(p)^*$. Here, $\alpha_p(\cD_P)$ is an ideal of $\cD_P$ and is thus hereditary. The Nica-covariance condition ensures that $\alpha_p(1)=V(p)V(p)^*$ and $\alpha_q(1)=V(q)V(q)^*$ commute (their product is either $0$ or $\alpha_{p\vee q}(1)$). 
\end{example} 

\begin{example} Fix $k\geq 2$ and let $\cA$ be the UHF algebra of type $k^\infty$. Let $\alpha_1(a)=e_{11}\otimes a$, which is a $\ast$-endomorphism on $\cA$. This defines an injective semigroup dynamical system $(\cA, \mathbb{N}, \alpha)$, which is related to the celebrated Cuntz algebra $\cO_k$ \cite{Cuntz1977}. Unlike the previous two examples, $\alpha_k(\cA)$ is not an ideal of $\cA$ (since the UHF algebra is simple). However, one can check that the image $\alpha_1(\cA)$ is hereditary. Moreover, orthogonal projections $\{\alpha_n(1)\}$ clearly commute since $\alpha_n(1)\alpha_m(1)=\alpha_{\max\{n,m\}}(1).$  
\end{example} 

Having $\alpha_p(\cA)$ be hereditary subalgebras in $\cA$ allows us to construct a surjective left-inverse $\alpha_{p^{-1}}:\cA\to\cA$. For any $p\in P$, define 
\[\alpha_{p^{-1}}(a)=\alpha_p^{-1}(\alpha_p(1)a\alpha_p(1)).\]
Here, $\alpha_p(1)a\alpha_p(1)\in\alpha_p(\cA)$, and the injectivity of $\alpha_p$ guarantees that the inverse exists, so that $\alpha_{p^{-1}}$ is a well-defined map on $\cA$. Moreover, for any $a\in\cA$, $\alpha_p(a)=\alpha_p(1)\alpha_p(a)\alpha_p(1)$. Therefore, $\alpha_{p^{-1}}(\alpha_p(a))=a$, so that $\alpha_{p^{-1}}$ is a surjective left-inverse of $\alpha_p$. The map $\alpha_{p^{-1}}$ is usually not multiplicative, since for any $a,b\in \cA$, it is not necessarily true that $\alpha_p(1)ab\alpha_p(1)= \alpha_p(1)a\alpha_p(1)b\alpha_p(1)$. However, for $a\in\cA$ that commute with $\alpha_p(1)$, and for any $b\in \cA$, we have 
\[\alpha_p(1)ab\alpha_p(1)=\alpha_p(1)(\alpha_p(1)a)b\alpha_p(1)=\alpha_p(1)a\alpha_p(1)b\alpha_p(1).\]
Apply the inverse map $\alpha_p^{-1}$, we have $\alpha_{p^{-1}}(ab)=\alpha_{p^{-1}}(a)\alpha_{p^{-1}}(b)$. Similarly, if $b$ commutes with $\alpha_p(1)$, we also have  $\alpha_{p^{-1}}(ab)=\alpha_{p^{-1}}(a)\alpha_{p^{-1}}(b)$.

The following technical lemma is essential in our later calculations. 

\begin{lemma}\label{lm.compare} Let $(\cA,P,\alpha)$ be an injective semigroup dynamical system over an abelian semigroup $P$ where its projections are aligned. For any $p,q\in P$, let $H=\alpha_{p^{-1}}(\alpha_q(1))$. Then $H$ is an orthogonal projection, and $\alpha_q(1)\leq H$. 
\end{lemma} 

\begin{proof} Since the projections are aligned, $\alpha_p(1)$ and $\alpha_q(1)$ commute and \[\alpha_p(1)\alpha_q(1)=\alpha_q(1)\alpha_p(1)=\alpha_p(1)\alpha_q(1)\alpha_p(1).\] 
Therefore, $\alpha_p(H)=\alpha_p(H^*)=\alpha_p(H^2)$, and the injectivity of $\alpha_p$ ensures that $H=H^*=H^2$ is an orthogonal projection.

To compare $\alpha_q(1)$ and $H$, we can first compare $\alpha_p(\alpha_q(1))=\alpha_{pq}(1)$ with $\alpha_p(H)=\alpha_p(1)\alpha_q(1)$. 
Since $P$ is abelian,  $\alpha_{pq}(1)=\alpha_{qp}(1)$ and \[\alpha_{pq}(1)\alpha_p(H)=\alpha_p(\alpha_q(1))\alpha_p(1) \alpha_q(1)=\alpha_{qp}(1)\alpha_q(1)=\alpha_{pq}(1).\]
Because $\alpha_p$ is injective, apply the inverse map $\alpha_p^{-1}$, we obtain that $\alpha_q(1)\leq H$. 
\end{proof} 

\section{Representations of Semigroup Dynamical Systems}\label{sec.rep}

There are many ways to construct an operator algebra from a dynamical system. For $C^*$-algebras, the most common way is through the following definition. This coincides with Murphy's definition \cite{Murphy1996} of $C^*(\cA,P,\alpha)$ where he proved its universal property with respect to covariant representations.

\begin{definition} An \emph{isometric covariant representation} for $(\cA,P,\alpha)$ is a pair $(\pi, V)$ where,
\begin{enumerate}
\item $\pi:\cA\to\bh{H}$ is a unital $\ast$-homomorphism.
\item $V:P\to\bh{H}$ is an isometric representation with $V(e)=I$.
\item For any $p\in P$ and $a\in\cA$, we have 
\[V(p)\pi(a)V(p)^* =\pi(\alpha_p(a)).\]
\end{enumerate}

We denote the universal $C^*$-algebra for isometric covariant representations by $\cA\rtimes_\alpha^* P$. 
\end{definition} 

The $C^*$-algebra $\cA\rtimes_\alpha^* P$ arising from a semigroup dynamical system is very well studied. One may refer to \cite{Murphy1996, Laca2000, Larsen2010} for related constructions and discussions. 

\begin{remark} Since $P$ is abelian, it is automatically an Ore semigroup, which embeds inside a group $G$ with $G=P^{-1}P$. A result of Laca \cite{Laca2000} proves that isometric covariant representations can be further dilated to a unitary covariant representation of an automorphic dynamical system $(\cB,G,\beta)$, and $\cA\rtimes_\alpha^* P$ is a full corner of the group crossed product $\cB\rtimes_\beta G$. 
\end{remark} 

There is a non-self-adjoint analogue of this definition: by multiplying $V(p)$ on the right of the condition (3), the adjoint $V(p)^*$ is replaced by $V(p)$ on the right hand side.

\begin{definition} An \emph{isometric right covariant representation} for $(\cA,P,\alpha)$ is given by a pair $(\pi, V)$ where,
\begin{enumerate}
\item $\pi:\cA\to\bh{H}$ is a unital $\ast$-homomorphism.
\item $V:P\to\bh{H}$ is an isometric representation with $V(e)=I$.
\item For any $p\in P$ and $a\in\cA$, we have 
\[V(p)\pi(a) =\pi(\alpha_p(a))V(p).\]
\end{enumerate}
\end{definition} 

There is a dual definition of left covariance where the covariance condition is replaced by $\pi(a)W(p) =W(p)\pi(\alpha_p(a))$, which leads to a rich literature of non-self-adjoint operator algebras constructed using this alternative convariance condition. When $\cA$ is a $C^*$-algebra, this dual definition is equivalent to $W(p)^* \pi(a) = \pi(\alpha_p(a)) W(p)^*$. When $P$ is abelian, one can define $V(p)=W(p)^*$, which is also a representation of $P$. Consequently, an isometric right covariant representation $(\pi,V)$ corresponds to an co-isometric left covariant pair $(\pi,W)$ (that is, $W$ is an co-isometric representation of $P$ and $\pi(a) W(p)=W(p)\pi(\alpha_p(a))$). The non-self-adjoint operator algebra associated such co-isometric right covariant pair was considered in the context where the action $\alpha$ is automorphic \cite[Remark 3.3.2]{DFK2014} or unital \cite{Kakariadis2016}. One may refer to \cite{Peters1984, DFK2014} for more detailed discussion for operator algebras arising from this alternative covariance condition.

As pointed out by Peters \cite{Peters1984}, one concern for the left covariance relation is that it forces $\ker(\alpha_p)\subset\ker(\pi)$. This is not an issue here since the semigroup action $\alpha$ is injective.

One can associate a non-self-adjoint operator algebra which is universal for isometric right covariant representations. Define $c_{00}(P,\cA)$ to be the linear span of $a\otimes \delta_p$, with the multiplication 
\[(a\otimes \delta_p) \cdot (b\otimes \delta_q)=a\alpha_p(b) \otimes \delta_{pq}.\]
For any isometric right covariant representation $(\pi,V)$ on $\bh{H}$, it defines a homomorphism $\Phi_{\pi,V}:c_{00}(P,\cA)\to\bh{H}$ by $\Phi_{\pi,V}(a\otimes \delta_p)=\pi(a)V(p)$ and extend linearly. For each $n\geq 1$, we can similarly define a homomorphism $\Phi_{\pi,V}^{(n)}: \cM_n(c_{00}(P,\cA))\to\cB(\cH^n)$ by $\Phi_{\pi,V}^{(n)}([a_{i,j}]\otimes \delta_p)=[\pi(a_{i,j})V(p)]$. For $x\in\cM_n(c_{00}(P,\cA))$, define
\[\|x\|_n=\sup\{\|\Phi_{\pi,V}^{(n)}(x)\|: (\pi,V) \mbox{ is isometric right covariant}\}.\]
Here, we have to assume a priori that there exists at least one isometric right covariant representation to avoid taking the supremum over an empty set. This is rarely a concern for us, since isometric covariant representations are automatically isometric right covariant and all the examples we considered have a well-defined universal $C^*$-algebra $\cA\rtimes^*_\alpha P$ for isometric covariant representations. 
A faithful isometric covariant representation of $\cA\rtimes^*_\alpha P$ will also assign a non-zero value for $\|x\|_n$ except when $x=0$. 
Also notice that since $V$ is isometric and $\pi$ is $\ast$-homomorphic, for $x=\sum_{i=1}^k A_i\otimes \delta_{p_i}$, we have $\|\Phi_{\pi,V}^{(n)}(x)\|\leq \sum_{i=1}^k \|\pi^{(n)}(A_i)\|<\infty$, so that the supremum indeed exists. This family of matrix norms $\|\cdot\|_n$ puts a non-self-adjoint operator algebra structure over $c_{00}(P,\cA)$, and we denote $\cA\rtimes_\alpha^{iso} P$ the (non-self-adjoint) norm closure of $c_{00}(P,\cA)$ with respect to this universal norm.  

We already mentioned that an isometric covariant representation $(\pi,V)$ is automatically an isometric right covariant representation. However, the converse is generally false, and it requires one extra condition. 

\begin{proposition}\label{prop.condition} Let $\pi:\cA\to\bh{H}$ be a unital $\ast$-homomorphism and $V:P\to\bh{H}$ be an isometric representation. Then the following are equivalent:
\begin{enumerate}
\item For all $p\in P$ and $a\in \cA$, \[V(p)\pi(a)V(p)^*=\pi(\alpha_p(a)).\]
\item For all $p\in P$ and $a\in \cA$, 
\[V(p)\pi(a)=\pi(\alpha_p(a))V(p),\]
and $V(p)V(p)^*=\pi(\alpha_p(1))$. 
\end{enumerate}
\end{proposition} 

\begin{proof} Assuming (1), since $V(p)^*V(p)=I$, 
\[V(p)\pi(a)=V(p)\pi(a)V(p)^*V(p)=\pi(\alpha_p(a))V(p).\]

Moreover, set $a=1\in \cA$, 
\[V(p)V(p)^*=V(p)\pi(1)V(p)^*=\pi(\alpha_p(1)).\]

Conversely, for any $p\in P$ and $a\in \cA$,
\begin{align*}
V(p)\pi(a)V(p)^* &= \pi(\alpha_p(a))V(p)V(p)^* \\
&= \pi(\alpha_p(a))\pi(\alpha_p(1)) \\
&= \pi(\alpha_p(a\cdot 1))=\pi(\alpha_p(a)). \qedhere
\end{align*}
\end{proof} 

\begin{example} In general, there are always isometric right covariant representations that are not isometric covariant.  Take an isometric covariant representation $(\rho, W)$ of $(\cA,P,\alpha)$ on $\bh{H}$. Let $\ell^2(P)$ be the Hilbert space with orthonormal basis $\{e_p:p\in P\}$. Define $\pi:\cA\to\cB(\cH\otimes \ell^2(P))$ by $\pi(a)=\rho(a)\otimes I_{\ell^2(P)}$. Define $V:P\to\cB(\cH\otimes\ell^2(P))$ by $V(p)=W(p) \otimes \lambda_p$, where $\lambda_p e_q=e_{pq}$ is the left-regular representation of $P$ on $\ell^2(P)$. We claim that $(\pi,V)$ is an isometric right covariant representation of $(\cA,P,\alpha)$, but it is not isometric covariant. It is clear that $\pi$ is a unital $*$-homomorphism and $V$ is an isometric representation with $V(e)=I$. For any $a\in \cA$ and $p\in P$, 
\begin{align*}
V(p)\pi(a)&=(W(p)\otimes \lambda_p)(\rho(a)\otimes I_{\ell^2(P)}) \\
&= (W(p)\rho(a)) \otimes \lambda_p \\
&= (\rho(\alpha_p(a)) W(p))\otimes \lambda_p \\
&= \pi(\alpha_p(a)) V(p).
\end{align*}
However, $V(p)V(p)^*=W(p)W(p)^*\otimes\lambda_p\lambda_p^*$ and $\pi(\alpha_p(1))=\rho(\alpha_p(1))\otimes I_{\ell^2(P)}$. Even though $W(p)W(p)^*=\rho(\alpha_p(1))$, $\lambda_p\lambda_p^*\neq I_{\ell^2(P)}$ unless $p$ is invertible in $P$. Therefore, whenever $P$ contains an element $a\in P$ without an inverse, $(\pi,V)$ is an isometric right covariant representation but not isometric covariant. 
\end{example}
In the case when $V(p)V(p)^*\neq \pi(\alpha_p(1))$ for an isometric right covariant representation, the orthogonal projection $V(p)V(p)^*$ is in fact always dominated by $\pi(\alpha_p(1))$:

\begin{lemma}\label{lm.proj.ineq} Let $(\pi,V)$ be an isometric right covariant representation. Then for any $p\in P$, $V(p)V(p)^*\leq \pi(\alpha_p(1))$.
\end{lemma}

\begin{proof} Since $\pi$ is unital, 
\begin{align*}
V(p)V(p)^* &= V(p)\pi(1)V(p)^* \\
&= \pi(\alpha_p(1))V(p)V(p)^*. 
\end{align*}
Since $V(p)V(p)^*$ and $\pi(\alpha_p(1))$ are orthogonal projections, this implies the desired inequality. \end{proof}

The following lemma is useful in many computations. 

\begin{lemma}\label{lm.right.cov} Let $(\pi,V)$ be an isometric right covariant representation for $(\cA,P,\alpha)$. Then for any $p\in P$ and $a\in \cA$, 
\[\pi(a) V(p) = \pi(a\alpha_p(1)) V(p).\]
If $a$ commutes with $\alpha_p(1)$, we have
\[\pi(a) V(p) = \pi(a\alpha_p(1)) V(p)=V(p) \pi(\alpha_{p^{-1}}(a)),\]
and,
\[V(p)^* \pi(a)  = \pi(\alpha_{p^{-1}}(a))V(p)^*.\]
\end{lemma} 

\begin{proof}
Since $\pi$ is unital, $V(p)=V(p)\pi(1)=\pi(\alpha_p(1))V(p)$. Multiplying $\pi(a)$ on the left proves the first equality. If in addition, $a$ commutes with $\alpha_p(1)$, then $a\alpha_p(1)=\alpha_p(1)a\alpha_p(1)\in\alpha_p(\cA)$. In this case, 
\[\pi(a) V(p) = \pi(a\alpha_p(1))V(p)=\pi(\alpha_p(\alpha_{p^{-1}}(a)))V(p)=V(p)\pi(\alpha_{p^{-1}}(a)).\]
Moreove, $a^*$ also commutes with $\alpha_p(1)$. Putting $a^*$ in the above equality and taking the adjoint on both sides give the last equality. 
\end{proof} 

\section{Dilation and C*-envelope}\label{sec.diln}

Let $\cB$ be an operator algebra,  not necessarily self-adjoint. A $C^*$-cover of $\cB$ is a pair $(C,j)$ where $C$ is a $C^*$-algebra and $j$ is a completely isometric homomorphism $j:\cB\to C$ such that $C=C^*(j(B))$. The $C^*$-envelope of $\cB$, denoted by $(C^*_{env}(\cB),i)$, is the smallest $C^*$-cover in the sense that for any $C^*$-cover $(C,j)$, there exists a $\ast$-homomorphism $\Phi:C\to C^*_{env}(\cB)$ such that $i=\Phi\circ j$.  
 
There are many ways to compute the $C^*$-envelope of an operator algebra. This paper focuses on the one using dilation theory. For a completely contractive homomorphism $\phi:\cB\to\bh{H}$, a dilation $\rho:\cB\to\bh{K}$ is a completely contractive homomorphism on a larger Hilbert space $\cK\supset\cH$ such that for all $a\in \cB$, 
\[\phi(a)=P_\cH \rho(a)\bigg|_\cH.\]
A result of Sarason \cite{Sarason1966} shows that $\cK$ can be decomposed as $\cK=\cH^+\oplus\cH\oplus \cH^-$, with respect to which $\rho(a)$ is represented by the following operator matrix:
\[\rho(a)=\begin{bmatrix} * & 0 & 0 \\ * & \phi(a) & 0 \\ * & * & *\end{bmatrix}\] 
A homomorphism $\phi$ is called maximal if any dilation $\rho$ can be written as $\phi\oplus \psi$ for some representation $\psi$. It turns out that the $C^*_{env}(\cB)$ is the $C^*$-algebra generated by a maximal completely isometric homomorphism of $\cB$. Moreover, a result of Dritschel and McCullough \cite{DM2005} shows that any completely contractive homomorphism $\phi$ of $\cB$ can be dilated to a maximal representation, and thus the $C^*_{env}(\cB)$ always exists. 

We start by pointing out that isometric covariant representations are maximal isometric right covariant representations of $\cA\rtimes_\alpha^{iso} P$. 

\begin{proposition}\label{prop.maximal} Let $(\pi,V)$ be an isometric covariant representation for $(\cA,P,\alpha)$. Then $(\pi,V)$ is also an isometric right covariant representation, and $(\pi,V)$ is maximal for $\cA\rtimes_{\alpha}^{iso} P$. 
\end{proposition}

\begin{proof} Since $(\pi,V)$ is isometric covariant, Proposition \ref{prop.condition} proves that it is also an isometric right covariant representation. 

To see it is maximal: let $(\rho,W)$ be a dilation of $(\pi,V)$ on a larger Hilbert space $\cK\supset \cH$. First of all, since $V(p)$ is isometric on $\cH$, $\cH$ must be invariant for $W$. Moreover, $\pi$ is a $\ast$-homomorphism, so that its dilation $\rho$ must be in the form of $\rho=\pi\oplus\phi$ for some unital $\ast$-homomorphism $\phi:\cA\to\cB(\cH^\perp)$ (for example, one can see this from Sarason's decomposition: $\rho(a^*)=\rho(a)^*$ so that the off-diagonal entries must be $0$). Therefore, with respect to $\cK=\cH\oplus \cH^\perp$, for each $p\in P$, 
\[W(p)=\begin{bmatrix} V(p) & A_p \\ 0 & B(p) \end{bmatrix},\]
and for each $a\in \cA$, 
\[\rho(a)=\begin{bmatrix} \pi(a) & 0 \\ 0 & \phi(a) \end{bmatrix}.\]
It suffices to prove that $A_p=0$ for all $p\in P$ so that $(\rho,W)=(\pi,V)\oplus(\phi,B)$. Since $(\rho,W)$ is an isometric right covariant representation, Lemma \ref{lm.proj.ineq} implies that $W(p)W(p)^*\leq \rho(\alpha_p(1))$. We have,
\begin{align*}
W(p)W(p)^* &= \begin{bmatrix} V(p)V(p)^* + A_pA_p^* & * \\ * & * \end{bmatrix} \\
&\leq \rho(\alpha_p(1)) \\
&= \begin{bmatrix} \pi(\alpha_p(1)) & 0 \\ 0 & * \end{bmatrix}.
\end{align*}

Therefore, $V(p)V(p)^* + A_pA_p^*\leq \pi(\alpha_p(1))$. However, $V(p)V(p)^*=\pi(\alpha_p(1))$ and $A_pA_p^*\geq 0$. This implies that $A_pA_p^*=0$ and thus $A_p=0$ for all $p\in P$. 
\end{proof}

\section{Main Result}\label{sec.main}

The goal of this section is to prove the following main result:

\begin{theorem}\label{thm.envelope} Let $(\cA,P,\alpha)$ be an injective semigroup dynamical system over an abelian semigroup $P$, with its projections aligned and each $\alpha_p(\cA)$ being hereditary. Then the $C^*$-envelope of $\cA\rtimes_\alpha^{iso} P$ is $\cA\rtimes_\alpha^* P$.
\end{theorem}

Proposition \ref{prop.maximal} proves that isometric covariant representations are maximal representations for the non-self-adjoint operator algebra $\cA\rtimes_\alpha^{iso} P$. It suffices to prove the converse that an isometric right covariant representation is not maximal if it is not isometric covariant. This is done through an explicit construction of a non-trivial dilation of an isometric right covariant representation that is not isometric covariant, using a dilation technique from \cite{Jaspar2016}. Consequently, we can invoke the result of Dritschel and McCullough \cite{DM2005} to show that the universal $C^*$-algebra with respect to isometric covariant representations is the $C^*$-envelope of $\cA\rtimes_\alpha^{iso} P$. 

Let $(\pi,V)$ be an isometric right covariant representation of $(\cA,P,\alpha)$. By Proposition \ref{prop.condition}, it is an isometric covariant representation if and only if $V(p)V(p)^*=\pi(\alpha_p(1))$ for all $p\in P$. Suppose $(\pi,V)$ is not isometric covariant, there must be a $q\in P$ with $V(q)V(q)^*\neq \pi(\alpha_q(1))$. By Lemma \ref{lm.proj.ineq}, the range projection $V(q)V(q)^*$ is strictly less than the projection $\pi(\alpha_q(1))$. Let $\cL_1=(\pi(\alpha_q(1))-V(q)V(q)^*)\cH$ be the subspace corresponding to their difference. Define $T:\cL_1\to\cH$ to be the inclusion map of $\cL_1$ inside $\cH$. In other words, $T$ is essentially the projection $\pi(\alpha_q(1))-V(q)V(q)^*$ with its domain restricted to $\cL_1$. We have $T^*T=I_{\cL_1}\in\cB(\cL_1)$ and $TT^*=\pi(\alpha_q(1))-V(q)V(q)^* \in \bh{H}$.

\begin{lemma}\label{lm.orthogonal} $V(q)^*T=T^*V(q)=0$
\end{lemma}

\begin{proof} By construction, the range of $T$ is $\pi(\alpha_q(1))\cH\ominus V(q)\cH$, which is orthogonal to the range of $V(q)$. Therefore, $TT^* V(q)V(q)^*=0$, and the result follows immediately. 
\end{proof}

\begin{lemma}\label{lm.PT} Let $P\in\bh{H}$ be any orthogonal projection so that $\pi(\alpha_q(1))\leq P$. Then $PT=T$. 
\end{lemma} 

\begin{proof} From the construction, $TT^*$ is the range projection onto $\pi(\alpha_q(1))\cH\ominus V(q)\cH$, which is contained in the range of $P$. Therefore,  $TT^*\leq P$ and thus $PT=T$.
\end{proof} 

Now, define maps $\pi_1$ and $V_1$ on $\cH_1=\cH\oplus \cL_1$ by 

\begin{align*}
\pi_1(a) &= \begin{bmatrix} \pi(a) & 0 \\ 0 & T^* \pi(\alpha_q(a)) T \end{bmatrix}, \\
V_1(p) &= \begin{bmatrix} V(p) & V(q)^* V(p) T \\ 0 & T^* V(p) T \end{bmatrix}.
\end{align*}

\begin{proposition}\label{prop.dilation.base} The pair $(\pi_1,V_1)$ is also an isometric right covariant representation.
\end{proposition} 

\begin{proof} We shall divide the proof into several claims:

\textbf{Claim 1}. $\pi_1$ is a unital $\ast$-homomorphism.

\textit{Proof of Claim 1}. $\pi_1=\pi\oplus\phi_1$ where $\phi_1(a)=T^*\pi(\alpha_q(a)) T:\cA\to\cB(\cL_1)$. It suffices to prove that $\phi_1$ is a unital $\ast$-homomorphism. First, $\phi_1(1)=T^* \pi(\alpha_q(1)) T$. By Lemma \ref{lm.PT}, $\pi(\alpha_q(1)) T=T$ and thus $\phi_1(1)=T^*T=I_{\cL_1}$. Now, for any $a,b\in \cA$, 
\begin{align*}
\phi_1(a)\phi_1(b) &= T^* \pi(\alpha_q(a)) TT^* \pi(\alpha_q(b)) T \\
&= T^* \pi(\alpha_q(a)) (\pi(\alpha_q(1))-V(q)V(q)^*) \pi(\alpha_q(b)) T \\
&= T^*  \pi(\alpha_q(ab)) T - T^* \pi(\alpha_q(a))V(q)V(q)^*\pi(\alpha_q(b)) T \\
&= \phi_1(ab) - T^* V(q) \pi(ab) V(q)^* T  \\
&= \phi_1(ab).
\end{align*}
Here, the last equality follows from Lemma \ref{lm.orthogonal} where $T^* V(q)=0$. $\phi_1$ is clearly $\ast$-linear. Therefore, it is a unital $\ast$-homomorphism, and so is $\pi_1$. 

\textbf{Claim 2}. $V_1$ is an isometric representation of $P$, and $V_1(e)=I$.

\textit{Proof of Claim 2}. By the definition, 
\[V_1(e)=\begin{bmatrix} V(e) & V(q)^* T \\ 0 & T^* T\end{bmatrix}=\begin{bmatrix} I_\cH & 0 \\ 0 & I_{\cL_1} \end{bmatrix}=I_{\cH\oplus \cL_1}\]
Here, $V(q)^* T=0$ follows from Lemma \ref{lm.orthogonal}.

We first prove that $V_1(p)$ is an isometry for all $p\in P$. Consider $V_1(p)^* V_1(p)$:
\begin{align*}
V_1(p)^* V_1(p) &= \begin{bmatrix} V(p)^* & 0 \\ T^* V(p)^* V(q) & T^* V(p)^* T\end{bmatrix} \begin{bmatrix} V(p) & V(q)^* V(p)T \\ 0 & T^*V(p)T \end{bmatrix} \\
&= \begin{bmatrix} V(p)^*V(p) & V(p)^*V(q)^*V(p)T \\
T^* V(p)^* V(q) V(p) & T^*V(p)^* V(q) V(q)^* V(p) T + T^* V(p)^* TT^* V(p) T \end{bmatrix}.
\end{align*}

Since $V(p)$ is an isometry, $V(p)^*V(p)=I$. Since $P$ is abelian, $V(p),V(q)$ commute, and together with Lemma \ref{lm.orthogonal},
\[T^* V(p)^* V(q) V(p)=T^* V(p)^* V(p) V(q) =T^* V(q)=0.\]
Similarly, $V(p)^*V(q)^*V(p)T=0$ as well. Lastly, we use $TT^*=\pi(\alpha_q(1))-V(q)V(q)^*$, and thus 
\begin{align*}
&T^*V(p)^* V(q) V(q)^* V(p) T + T^* V(p)^* TT^* V(p) T  \\
=& T^*V(p)^* V(q) V(q)^* V(p) T + T^* V(p)^* (\pi(\alpha_q(1))-V(q)V(q)^*) V(p) T \\
=& T^* V(p)^* \pi(\alpha_q(1)) V(p) T \\
=& T^* V(p)^* V(p) \pi(\alpha_{p^{-1}}(\alpha_q(1)))  T \\
=& T^* \pi(\alpha_{p^{-1}}(\alpha_q(1)))  T.
\end{align*}

By Lemma \ref{lm.compare}, $\pi(\alpha_q(1))\leq \pi(\alpha_{p^{-1}}(\alpha_q(1)))$, which, by Lemma \ref{lm.PT}, implies that $T^*\pi(\alpha_{p^{-1}}(\alpha_q(1)))  T=T^*T=I_{\cL_1}$.

Now it suffices to prove $V_1(p)V_1(r)=V_1(pr)$ for all $p,r\in P$. We have,
\begin{align*}
V_1(p)V_1(r) &= \begin{bmatrix} V(p) & V(q)^*V(p)T \\ 0& T^* V(p)T\end{bmatrix} \begin{bmatrix} V(r) & V(q)^*V(r)T \\ 0& T^* V(r)T\end{bmatrix} \\
&= \begin{bmatrix} V(pr) & V(p)V(q)^* V(r)T+V(q)^*V(p)TT^* V(r) T \\
0 & T^* V(p) TT^* V(r) T \end{bmatrix}. 
\end{align*}
Here,
\begin{align*}
 & V(p)V(q)^* V(r)T+V(q)^*V(p)TT^* V(r) T \\
=& V(q)^* V(q) V(p) V(q)^* V(r) T + V(q)^* V(p) (\pi(\alpha_q(1))-V(q)V(q)^*) V(r) T \\
=&  V(q)^*  V(p) V(q)V(q)^* V(r) T+ V(q)^* V(p)\pi(\alpha_q(1))V(r) T -V(q)^* V(p)V(q)V(q)^* V(r) T \\
=& V(q)^* V(p)\pi(\alpha_q(1))V(r) T \\
=& V(q)^* V(pr) \pi(\alpha_{r^{-1}}(\alpha_q(1))) T \\
=&  V(q)^* V(pr) T.
\end{align*}
Here, we applied Lemma \ref{lm.compare} and Lemma \ref{lm.PT} to show that $ \pi(\alpha_{r^{-1}}(\alpha_q(1)))\geq \pi(\alpha_q(1))$ and $\pi(\alpha_{r^{-1}}(\alpha_q(1))) T=T$.

Finally,
\begin{align*}
 & T^* V(p) TT^* V(r) T \\
=&T^* V(p) \pi(\alpha_q(1)) V(r) T - T^* V(p) V(q)V(q)^* V(r) T \\
=& T^* V(pr) \pi(\alpha_{r^{-1}}(\alpha_q(1))) T - (T^* V(q)) V(p)V(q)^*V(r)T \\
=& T^* V(pr)T.
\end{align*}

Therefore,
\[V_1(p)V_1(r) = \begin{bmatrix} V(pr) & V(q)^* V(pr) T \\ 0 & T^* V(pr) T^*\end{bmatrix} = V_1(pr).\]

\textbf{Claim 3}. For any $p\in P$ and $a\in \cA$, $V_1(p)\pi_1(a)=\pi_1(\alpha_p(a)) V_1(p)$. 

\textit{Proof of Claim 3}. From the definition of $V_1,\pi_1$, the left hand side becomes:
\[
V_1(p)\pi_1(a) = \begin{bmatrix} V(p)\pi(a) & V(q)^* V(p) TT^* \pi(\alpha_q(a)) T \\
0 & T^* V(p) TT^* \pi(\alpha_q(a)) T 
\end{bmatrix}.\]
The right hand side becomes:
\[\pi_1(\alpha_p(a))V_1(p) =\begin{bmatrix}
\pi(\alpha_p(a))V(p) & \pi(\alpha_p(a)) V(q)^* V(p) T \\
0 & T^* \pi(\alpha_{qp}(a)) TT^* V(p) T
\end{bmatrix}.
\]
First of all, the covariance condition implies $ V(p)\pi(a) =\pi(\alpha_p(a))V(p)$. Next,
\begin{align*}
&V(q)^* V(p) TT^* \pi(\alpha_q(a)) T\\
=& V(q)^* V(p) (\pi(\alpha_q(1))-V(q)V(q)^*)\pi(\alpha_q(a)) T \\
=& V(q)^* V(p) \pi(\alpha_q(a)) T -  V(q)^* V(p)V(q)V(q)^*\pi(\alpha_q(a)) T \\
=& V(q)^* \pi(\alpha_{pq}(a)) V(p) T - V(p) \pi(a) V(q)^* T \\
=&  \pi(\alpha_{p}(a)) V(q)^* V(p) T.
\end{align*}
Finally,
\begin{align*}
 & T^* V(p) TT^* \pi(\alpha_q(a)) T \\
=& T^* V(p) \pi(\alpha_q(a)) T -  T^* V(p) V(q)V(q)^* \pi(\alpha_q(a)) T \\
=& T^* \pi(\alpha_{pq}(a)) V(p) T - T^*V(q) V(p)V(q)^* \pi(\alpha_q(a)) T \\
=& T^* \pi(\alpha_{pq}(a)) V(p) T.
\end{align*}
On the other hand,
\begin{align*}
 & T^* \pi(\alpha_{qp}(a)) TT^* V(p) T \\
=& T^* \pi(\alpha_{qp}(a)) \pi(\alpha_q(1)) V(p) T -  T^* \pi(\alpha_{qp}(a)) V(q)V(q)^* V(p) T \\
=& T^* \pi(\alpha_{pq}(a)) V(p) T - T^* V(q) \pi(\alpha_{p}(a)) V(q)^* V(p) T \\
=& T^* \pi(\alpha_{pq}(a)) V(p) T.
\end{align*}
Therefore, all entries of $V_1(p)\pi_1(a)$ and $\pi_1(\alpha_p(a))V_1(p)$ match. \end{proof} 

The isometric right covariant representation $(\pi_1,V_1)$ is a non-trivial dilation of $(\pi,V)$. Indeed, by construction,
\[V_1(q) = \begin{bmatrix} V(p) & T \\ 0 & T^*V(q)T \end{bmatrix} = \begin{bmatrix} V(p) & T \\ 0 & 0 \end{bmatrix}.\]
The Hilbert space $\cH$ is clearly not a reducing subspace for $V_q(1)$. Therefore, $(\pi,V)$ is not maximal. Combined with Proposition \ref{prop.maximal}, we obtain:

\begin{corollary}\label{cor.maximal} An isometric right covariant representation is maximal if and only if it is isometric covariant. 
\end{corollary} 

Using the technique of Dritschel and McCullough \cite{DM2005} , Theorem \ref{thm.envelope} is an immedate consequence of Corollary \ref{cor.maximal}.

Recall that a non-self-adjoint operator algebra $\cS$ is called \emph{hyperrigid} \cite{Arveson2011} if for every $*$-representation $\pi$ of $C^*_{env}(S)$, the restriction of $\pi$ on $S$ has the unique extension property, that is: there exists a unique unital completely positive map $\phi$ on $C^*_{env}(S)$ that agrees with $\pi$ on $S$. It is known that having the unique extension property is equivalent to being a maximal representation. By Theorem \ref{thm.envelope} and Proposition \ref{prop.maximal}, we have that $\cA\rtimes_\alpha^{iso} P$ is hyperrigid. This is kindly pointed out by Evgenios Kakariadis.  

\begin{corollary} For an injective semigroup dynamical system $(\cA,P,\alpha)$ over an abelian semigroup $P$ with its projections aligned and each $\alpha_p(\cA)$ being hereditary, $\cA\rtimes_\alpha^{iso} P$ is hyperrigid
\end{corollary} 

For the rest of this section, we shall briefly go over the dilation from an isometric right covariant representation to an isometric covariant representation in the case when $P$ is countable. One may note that this process is essentially an explicit reconstruction of the abstract dilation result of Dritschel and McCullough \cite[Theorem 2.1]{DM2005}.

\begin{lemma}\label{lm.preservation} For any $p\in P$. If $V(p)V(p)^*=\pi(\alpha_p(1))$, then $V_1(p)V_1(p)^*=\pi_1(\alpha_p(1))$. 
\end{lemma} 

\begin{proof} We have,
\[V_1(p)V_1(p)^* = \begin{bmatrix} V(p)V(p)^* + V(q)^*V(p)TT^* V(p)^* V(q) & V(q)^* V(p)TT^* V(p)^* T \\
T^* V(p) TT^* V(p)^* V(q) & T^* V(p) TT^* V(p)^* T 
\end{bmatrix}. \]
Since $V(p)V(p)^*=\pi(\alpha_p(1))$, we have
\begin{align*}
 &V(q)^*V(p)TT^* V(p)^* V(q)\\
=& V(q)^*V(p) \pi(\alpha_q(1))  V(p)^* V(q) - V(q)^*V(p) V(q)V(q)^* V(p)^* V(q) \\
=& V(q)^* \pi(\alpha_{pq}(1)) V(p)V(p)^* V(q) - V(q)^*V(q)V(p)V(p)^*V(q)^*V(q) \\
=& V(q)^* \pi(\alpha_{pq}(1)) \pi(\alpha_p(1)) V(q) -  \pi(\alpha_p(1))  \\
=& \pi(\alpha_{p}(1)) V(q)^* V(q) -  \pi(\alpha_p(1)) =0.
\end{align*}
And,
\begin{align*}
 & V(q)^* V(p)TT^* V(p)^* T\\
=& V(q)^* V(p) \pi(\alpha_q(1)) V(p)^* T - V(q)^* V(p) V(q)V(q)^* V(p)^* T \\
=& V(q)^* \pi(\alpha_{pq}(1)) T - V(p)V(p)^* V(q)^* T \\
=& \pi(\alpha_p(1))V(q)^* T  = 0.
\end{align*}
Similarly, $T^* V(p) TT^* V(p)^* V(q)=0$. Finally,
\begin{align*}
 &T^* V(p) TT^* V(p)^* T\\
=& T^* V(p) \pi(\alpha_q(1)) V(p)^* T - T^* V(p) V(q)V(q)^* V(p)^* T \\
=& T^*  \pi(\alpha_{pq}(1))V(p) V(p)^* T - T^*V(q) V(p)V(p)^* V(q)^* T \\
=&  T^*  \pi(\alpha_q(\alpha_p(1))) T.
\end{align*}
Therefore, $V_1(p)V_1(p)^*=\pi_1(\alpha_p(1))$. 
\end{proof} 

Now $V(q)$ is dilated to 
\[V_1(q)=\begin{bmatrix} V(q) & V(q)^* V(q) T \\ 0& T^* V(q) T \end{bmatrix} = \begin{bmatrix} V(q) & T \\ 0 & 0 \end{bmatrix}.\]

\begin{lemma}\label{lm.restriction} We have \[\pi_1(\alpha_q(1))-V_1(q)V_1(q)^*=I_{\cL_1}.\]
In particular, $(\pi_1(\alpha_q(1))-V_1(q)V_1(q)^*)\cH=\{0\}$.
\end{lemma} 

\begin{proof} We have
\begin{align*}
\pi_1(\alpha_q(1))-V_1(q)V_1(q)^* &= \begin{bmatrix} \pi(\alpha_q(1)) - V(q)V(q)^* - TT^* & 0 \\ 0 & T^* \pi(\alpha_q(1)) T \end{bmatrix} \\
&= \begin{bmatrix} 0 & 0 \\ 0 & T^* \pi(\alpha_q(1)) T \end{bmatrix} \\
&= \begin{bmatrix} 0 & 0 \\ 0 & I_{\cL_1} \end{bmatrix}
\end{align*}
Therefore, $\pi_1(\alpha_q(1))-V_1(q)V_1(q)^*=I_{\cL_1}$, and since $\cH=\cL_{1}^\perp$ in $\cH_1=\cH\oplus\cL_1$, $(\pi_1(\alpha_q(1))-V_1(q)V_1(q)^*)\cH=\{0\}$
\end{proof} 

\begin{remark} It is not true that $\pi_1(\alpha_q(1))=V_1(q)V_1(q)^*$. However, the difference $(\pi(\alpha_p(1))-V(p)V(p)^*)\cH\subset\cH$ is being pushed to $(\pi_1(\alpha_q(1))-V_1(q)V_1(q)^*)\cH_1=\cL_1=\cH^\perp$. We can now repeat this process and take an inductive limit to get rid of this difference. 
\end{remark} 

\begin{proposition}\label{prop.dilation1} Let $(\pi,V)$ be an isometric right covariant representation on $\bh{H}$ where $V(q)V(q)^*\neq \pi(\alpha_q(1))$ for some $q\in P$. Then there exists an isometric right covariant dilation $(\rho,W)$ of $(\pi,V)$ such that $W(q)W(q)^*=\rho(\alpha_q(1))$. 

Moreover, if $V(p)V(p)^*=\pi(\alpha_p(1))$ for some $p\in P$, then $W(p)W(p)^*=\rho(\alpha_p(1))$.
\end{proposition} 

\begin{proof} We start with $\cH_0=\cH$ and $\pi_0=\pi, V_0=V$. When $V_n(q)V_n(q)^*\neq \pi_n(\alpha_q(1))$, we let $\cL_{n+1}=(\pi_n(\alpha_q(1))-V_n(q)V_n(q)^*)\cH_n$ be the subspace in $\cH_n$, and let $T_n:\cL_{n+1}\to\cH_n$ be the inclusion map. Define a dilation $(\pi_{n+1},V_{n+1})$ on $\cH_{n+1}=\cH_n\oplus\cL_{n+1}$ by 
\begin{align*}
\pi_{n+1}(a) &= \begin{bmatrix} \pi_n(a) & 0 \\ 0 & T^* \pi_n(\alpha_q(a)) T \end{bmatrix} ; \\
V_{n+1}(p) &= \begin{bmatrix} V_n(p) & V_n(q)^* V_n(p) T \\ 0 & T^* V_n(p) T \end{bmatrix}. 
\end{align*}
Proposition \ref{prop.dilation.base} shows that $(\pi_{n+1},V_{n+1})$ is also an isometric right covariant representation. Moreover, Lemma \ref{lm.restriction} implies that 
\[(\pi_{n+1}(\alpha_q(1))-V_{n+1}(q)V_{n+1}(q))\cH_n=\{0\}.\] 
Take the inductive limit of this process: we get a dilation $(\rho,W)$ on $\cK=\cH\oplus\left(\bigoplus_{n=1}^\infty \cL_n\right)=\displaystyle{\lim_{\longrightarrow}} \cH_n$. This is an isometric right covariant representation since each $(\pi_n, V_n)$ is. Moreover, $\rho(\alpha_q(1))=W(q)W(q)^*$ since $(\rho(\alpha_q(1))-W(q)W(q)^*)\cH_n=\{0\}$ and the Hilbert space $\cK$ is the inductive limit of $\cH_n$. 

Moreover, if $V(p)V(p)^*=\pi(\alpha_p(1))$ for some $p\in P$. Lemma \ref{lm.preservation} implies that $V_n(p)V_n(p)^*=\pi_n(\alpha_p(1))$, which is again preserved by the inductive limit. \end{proof}

Now we can repeatedly apply Proposition \ref{prop.dilation1} to dilate an isometric right covariant representation to an isometric covariant one. Since $P$ is countable, there exists countably many $q\in P$ with $V(q)V(q)^*\neq \pi(\alpha_q(1))$. Applying Proposition \ref{prop.dilation1} for $q$ will close the gap between $V(q)V(q)^*$ and $\pi(\alpha_q(1))$. Repeating this process at most countably many times will yield a dilation $(\rho,W)$ such that $W(p)W(p)=\rho(\alpha_p(1))$. This dilation is an isometric covariant representation.

\section{Functoriality}\label{sec.functor}

When an algebraic structure embeds inside another, it is often an interesting question on whether this embedding translates into an embedding of their associated operator algebras. This is known as the functoriality of the operator algebra. For example, in the case of algebraic integers in a number field, Cuntz, Deninger, and Laca proved that the $C^*$-algebra $\fT_R$ is functorial, that is when $R\subset S$ are rings of algebraic integers in a number field $K$, this embedding induces an injective $\ast$-homomorphism from the $ax+b$ semigroup $C^*$-algebra $\fT_R$ to $\fT_S$ \cite[Proposition 3.2, Theorem 4.13]{CDL2013}. Similar functoriality results were documented in semigroup $C^*$-algebras of $ax+b$ type semigroups from number fields \cite{XLi2012, Bruce2020}. On the other hand, quotients of semigroup $C^*$-algebras often fail to satisfy functoriality. For example, the map from $\cO_2=C^*(s_1,s_2)$ to $\cO_3=C^*(t_1,t_2,t_3)$ that sends $s_i$ to $t_i$ cannot be extended to a $*$-homomorphism; ring $C^*$-algebras do not satisfy functoriality in general \cite{CDL2013}. 

In this section, we would like to explore functoriality for both the $C^*$-algebra and the non-self-adjoint operator algebra associated with a semigroup dynamical system. 

\begin{definition}\label{df.extension} Let $(\cA,P,\alpha)$ be a semigroup dynamical system. We say a semigroup dynamical system $(\cB,Q,\beta)$ is \emph{an extension} of $(\cA,P,\alpha)$ if 
\begin{enumerate}
\item There exists a unital $\ast$-homomorphism $\iota:\cA\to\cB$. 
\item There exists a unital semigroup homomorphism $j:P\to Q$.
\item For any $p\in P\subset Q$ and $a\in \cA$, 
\[\iota(\alpha_p(a))=\beta_{j(p)}(\iota(a)).\]
\end{enumerate}

In other words, $(\cB,Q,\beta)$ is an extension of $(\cA,P,\alpha)$ if the following diagram commutes:
\begin{figure}[h]
    \centering
    \begin{tikzpicture}[scale=0.9]

    \node at (-2,2) {$\cA$};
    \node at (2,2) {$\cB$};
    \node at (-2,0) {$\cA$};
    \node at (2,0) {$\cB$};

    \draw[->] (-1.5,2) -- (1.5,2);
    \draw[->] (-1.5,0) -- (1.5,0);
	\draw[->] (-2,1.5) -- (-2,0.5);
    \draw[->] (2,1.5) -- (2,0.5);

	\node at (0, 2.3) {$\iota$};
	\node at (-2.3, 1) {$\alpha_p$};
	\node at (1.5, 1) {$\beta_{j(p)}$};
	\node at (0, 0.3) {$\iota$};
    \end{tikzpicture}
\end{figure}
\end{definition}

Throughout this section, we still assume that both $(\cA,P,\alpha)$ and $(\cB,Q,\beta)$ are injective semigroup dynamical systems where their respective projections are aligned and images of $\alpha_p$ and $\beta_q$ are hereditary. 

\begin{lemma}\label{lm.lift} Let $(\cB,Q,\beta)$ be an extension of $(\cA,P,\alpha)$ by the maps $\iota:\cA\to\cB$ and $j:P\to Q$. Let $\rho:\cB\to\bh{H}$ be a unital $*$-homomorphism and $W:Q\to\bh{H}$ be an isometric representation with $W(e)=I$. Define $\pi=\rho\circ \iota:\cA\to\bh{H}$ and $V=W\circ j:P\to\bh{H}$. 
\begin{enumerate}
\item If $(\rho,W)$ is an isometric covariant representation of $(\cB,Q,\beta)$, then $(\pi,V)$ is an isometric covariant representation of $(\cA,P,\alpha)$. 
\item If $(\rho,W)$ is an isometric right covariant representation of $(\cB,Q,\beta)$, then $(\pi,V)$ is an isometric right covariant representation of $(\cA,P,\alpha)$. 
\end{enumerate}
\end{lemma} 

\begin{proof} First of all, $\pi$ is a composition of two unital $*$-homomoprhisms so that it is also a unital $*$-homomorphism from $\cA$ to $\bh{H}$. Since $W$ is an isometric representation of $Q$, composing with the unital semigroup homomorphism $j$ yields an isometric representation $V$ of $P$ with $V(e)=W(e)=I$. 

For (1), when $(\rho, W)$ is isometric covariant, for any $p\in P$ and $a\in \cA$, we have 
\begin{align*}
V(p)\pi(a)V(p)^* &= W(j(p)) \rho(\iota(a)) W(j(p))^*  \\
&= \rho(\beta_{j(p)}(\iota(a))) \\
&=\rho(\iota(\alpha_p(a)))=\pi(\alpha_p(a)).
\end{align*} 

For (2), when $(\rho, W)$ is isometric right covariant, for any $p\in P$ and $a\in \cA$, we have 
\begin{align*}
V(p)\pi(a) &= W(j(p)) \rho(\iota(a))   \\
&= \rho(\beta_{j(p)}(\iota(a))) W(j(p)) \\
&=\rho(\iota(\alpha_p(a))) W(j(p)) =\pi(\alpha_p(a))V(p). \qedhere
\end{align*} 
\end{proof} 

\begin{proposition}\label{prop.functor.star} Let $(\cB,Q,\beta)$ be an extension of $(\cA,P,\alpha)$ by the maps $\iota:\cA\to\cB$ and $j:P\to Q$. Let $(\pi^u, V^u)$ be the universal isometric covariant representation for $(\cA,P,\alpha)$, and $(\rho^u, W^u)$ be the universal isometric covariant representation for $(\cB,Q,\beta)$. 

Then there exists a $*$-homomorphism $\phi: \cA\rtimes_\alpha^* P \to \cB\rtimes_\beta^* Q$ such that for any $a\in \cA$, $\phi(\pi^u(a))=\rho^u(\iota(a))$; for any $p\in P$, $\phi(V^u(p))=W^u(j(p))$. 
\end{proposition}

\begin{proof} Let $\pi=\rho^u\circ \iota$ and $V=W^u\circ j$. Lemma \ref{lm.lift} proves that $(\pi,V)$ is an isometric $^*$-covariant representation. By the universality of $(\pi^u, V^u)$, there exists a $*$-homomorphism $\phi$ such that for any $a\in \cA$, $\phi(\pi^u(a))=\rho^u(\iota(a))$; for any $p\in P$, $\phi(V^u(p))=W^u(j(p))$. The image of $\phi$ is inside $\cB\rtimes_\beta^* Q$ so that $\phi$ is the desired map.
\end{proof} 

\begin{proposition}\label{prop.functor.nsa} Let $(\cB,Q,\beta)$ be an extension of $(\cA,P,\alpha)$  by the maps $\iota:\cA\to\cB$ and $j:P\to Q$. Then the map $\psi:c_{00}(P,\cA)\to c_{00}(Q,\cB)$ defined by 
\[\psi(a\otimes\delta_p)=\iota(a)\otimes \delta_{j(p)},\]
extends to a completely contractive homomorphism from $\cA\rtimes^{iso}_\alpha P$ to $\cB\rtimes^{iso}_\beta Q$. 
\end{proposition}

\begin{proof} First of all, for any $a,b\in \cA$ and $p,q\in P\subset Q$, 
\begin{align*}
\psi((a\otimes\delta_p)(b\otimes \delta_q)) &= \psi(a\alpha_p(b)\otimes\delta_{pq}) \\
&= \iota(a\alpha_p(b))\otimes\delta_{j(pq)} \\
&= \iota(a) \beta_{j(p)}(\iota(b))\otimes \delta_{j(p)j(q)} \\
&= (\iota(a)\otimes\delta_{j(p)})(\iota(b)\otimes\delta_{j(q)}) \\
&= \psi(a\otimes\delta_p)\psi(b\otimes\delta_q).
\end{align*}

Hence, $\psi$ is a homomorphism. 

Take $x\in c_{00}(P,\cA)$ where 
\[x=\sum_{i=1}^k a_i\otimes \delta_{p_i}.\]

By definition, 
\[\psi(x)=\sum_{i=1}^k \iota(a_i)\otimes \delta_{j(p_i)}\in c_{00}(Q,\cB).\]

By the definition of the norm,
\[\|\psi(x)\| = \sup\{\|\sum_{i=1}^k \rho(\iota(a_i))W(j(p_i))\|: (\rho,W) \mbox{ is isometric right covariant}\}. \]

By Lemma \ref{lm.lift}, $(\rho\circ\iota, W\circ j)$ is an isometric right covariant representation for $(\cA,P,\alpha)$. Therefore,
\begin{align*}
\|\psi(x)\| &\leq  \sup\{\|\sum_{i=1}^k \pi(a_i)V(p_i)\|: (\pi,V) \mbox{ is isometric right covariant}\}
= \|x\|.
\end{align*}
Therefore, $\psi$ is contractive. One can apply a similar argument to show that $\psi$ is completely contractive, and thus extends to a completely contractive homomorphism from $\cA\rtimes_\alpha^{iso} P$ to $\cB\rtimes_\beta^{iso} Q$.
\end{proof} 

\begin{lemma}\label{lm.functor.ci} Let $\cA_1$ and $\cA_2$ be two operator algebras and $\psi:\cA_1\to\cA_2$ be a completely contractive homomorphism. Let $(C^*_{env}(\cA_k), i_k)$ be the $C^*$-envelopes for $\cA_k$, $k=1,2$, and let $\phi:C^*_{env}(\cA_1)\to C^*_{env}(\cA_2)$ be a $*$-homomorphism such that $\phi\circ i_1 = i_2 \circ \psi$. Then $\phi$ is injective if and only if $\psi$ is completely isometric.  
\end{lemma} 

\newpage 

\begin{proof} Consider the following commutative diagram:

\begin{figure}[h]
    \centering

    \begin{tikzpicture}[scale=0.9]

    \node at (-3,2) {$\cA_1$};
    \node at (3.8,2) {$C^*_{env}(\cA_1)$};
    \node at (-3,0) {$\cA_2$};
    \node at (3.8,0) {$C^*_{env}(\cA_2)$};

    \draw[->] (-2.6,2) -- (2.6,2);
    \draw[->] (-2.6,0) -- (2.6,0);
	\draw[->] (-3,1.5) -- (-3,0.5);
    \draw[->] (3.8,1.5) -- (3.8,0.5);

	\node at (0, 2.3) {$i_1$};
	\node at (3.5, 1) {$\phi$};
	\node at (-3.3, 1) {$\psi$};
	\node at (0, 0.3) {$i_2$};
    \end{tikzpicture}

\end{figure}

By the definition of $C^*$-envelope, $i_1$ and $i_2$ are completely isometric. If $\phi$ is an injective $*$-homomorphism, it must be an isometric $*$-homomorphism between two $C^*$-algebras, and thus also completely isometric. Therefore, $\phi\circ i_1=i_2\circ \psi$ is completely isometric, implying that $\psi$ is completely isometric. 

On the other hand, if $\phi$ is completely isometric, let $\cB=C^*(i_2\circ \psi(\cA_1))$, which is a $C^*$-subalgebra in $C^*_{env}(\cA_2)$. The pair $(\cB, i_2\circ\psi)$ is a $C^*$-cover for $\cA_1$. By the minimality of $C^*$-envelope, there exists a $*$-homomorphism $\Phi:\cB\to C^*_{env}(\cA_1)$ so that for any $a\in \cA_1$, $i_1(a)=\Phi(i_2(\psi(a)))$. For any $x=i_1(a)\in i_1(\cA_1)$, 
\[x=\Phi(i_2(\psi(a)))=\Phi(\phi(i_1(a)))=\Phi(\phi(x)).\]
Since $C^*(i_1(\cA_1))=C^*_{env}(\cA_1)$, $\Phi\circ\phi$ is the identity map on $C^*_{env}(\cA_1)$, and thus $\phi$ is injective. 
\end{proof} 

\begin{corollary}\label{cor.functor} Let $(\cB,Q,\beta)$ be an extension of $(\cA,P,\alpha)$ by the maps $\iota:\cA\to\cB$ and $j:P\to Q$. Let $\phi:\cA\rtimes_\alpha^* P\to\cB\rtimes_\beta^* Q$ be the $*$-homomorphism from Proposition \ref{prop.functor.star}, and $\psi:\cA\rtimes^{iso}_\alpha P\to\cB\rtimes^{iso}_\beta Q$ be the completely contractive homomorphism from Proposition \ref{prop.functor.nsa}. Then $\phi$ is injective if and only if $\psi$ is completely isometric. 
\end{corollary} 

\begin{proof} By Theorem \ref{thm.envelope}, $C^*_{env}(\cA\rtimes^{iso}_\alpha P)=\cA\rtimes_\alpha^* P$ and $C^*_{env}(\cB\rtimes^{iso}_\beta Q)=\cB\rtimes_\beta^* Q$. The rest follows immediately from Proposition \ref{prop.functor.star} and \ref{prop.functor.nsa}, and Lemma \ref{lm.functor.ci}.
\end{proof}

\section{Example from number fields and commutative rings}\label{sec.ex}

Let $R$ be a commutative ring and $M$ be a unital sub-semigroup of the commutative multiplicative semigroup $R^\times$. There is a rich literature of $C^*$-algebras associated with the $ax+b$ type semigroups $R\rtimes M$, where $(x,a)(y,b)=(x+ay,ab)$. The $C^*$-algebra often encodes the additive structure of the ring $R$, multiplicative structure of $M$, and certain ideal structure. In its simplest form, when $R=\mathbb{Z}$ and $M=\mathbb{N}^\times$, Cuntz considered the $C^*$-algebra $\cQ_\mathbb{N}$  \cite{Cuntz2000}. More recent developments include ring $C^*$-algebras \cite{CuntzLi2010, XLi2010}, $\mathbb{Z}\rtimes\langle 2\rangle$ \cite{LarsenLi2012}, $R\rtimes R^\times$ for arbitrary algebraic integers $R$ in a number field \cite{CDL2013}, $\mathbb{Z}\rtimes\langle S\rangle$ \cite{BOS2018}, and $R\rtimes R_{\mathfrak{m},\Gamma}$ for congruence monoids \cite{Bruce2020}. 

We assume that the multiplication of $M$ on $R$ is left-cancellative, that is for any $x,y\in R$ and $a\in M$, $ax=ay$ implies $x=y$. With this assumption, the semigroup $R\rtimes M$ is a left-cancellative semigroup: for $(x,a)(y,b)=(x,a)(z,c)$, one has $x+ay=x+az$ and $ab=ac$, which implies $y=z$ and $b=c$. Recent development of semigroup $C^*$-algebras by Xin Li allows us to study the $C^*$-algebras of left-cancellative semigroups. Here, we give a brief overview of Xin Li's construction of the universal semigroup $C^*$-algebras. 

\begin{definition} Let $P$ be a left-cancellative semigroup, and $I$ be a right ideal of $P$. Then for $p\in P$, define $pI=\{px: x\in I\}$ and $p^{-1}I=\{x: px\in I\}$. Then the set of constructible ideals $\cJ(P)$ is the smallest set of right ideals so that 
\begin{enumerate}
\item $\emptyset,P\in \cJ(P)$.
\item If $I\in \cJ(P)$, then $p\cdot I$ and $p^{-1} I\in \cJ(P)$. 
\item If $I,J\in \cJ(P)$, then $I\cap J\in \cJ(P)$. 
\end{enumerate}
\end{definition} 

In fact, one can explicitly describe $\cJ(P)$ as the set
\[\cJ(P)=\{q_n^{-1}p_n \cdots q_1^{-1} p_1 P: p_1,\cdots, p_n, q_1,\cdots, q_n \in P\}\cup\{\emptyset\}.\]

\begin{definition}\label{def.semigroupC} Let $P$ be a left-cancellative semigroup and $\cJ(P)$ be the set of its constructible ideals. Then the semigroup $C^*$-algebra $C^*(P)$ is the universal $C^*$-algebra generated by:
\begin{enumerate}
\item A set of isometries $\{v_p: p\in P\}$ where $v_p v_q=v_{pq}$ and $v_e=I$. 
\item A set of orthogonal projections $\{e_I: I\in\cJ(P)\}$ such that $e_P=I$ and $e_\emptyset=0$.
\item For any $p\in P$ and $I\in\cJ(P)$, 
\[v_p e_I v_p^* = e_{p I}, v_p^* e_I v_p = e_{p^{-1} I}.\]
\item For any $I,J\in \cJ(P)$, 
\[e_I e_J=e_{I\cap J}.\]
\end{enumerate}
\end{definition} 

The projection $e_I$ can be explicitly described by $v_p$. For an ideal \[I=q_n^{-1}p_n \cdots q_1^{-1} p_1 P,\]
we have
\[e_I=v_{q_n}^* v_{p_n} \cdots v_{q_1}^* v_{p_1} v_{p_1}^* v_{q_1} \cdots v_{p_n}^* v_{q_n}.\]
Moreover, the condition (4) is in fact redundant, as it follows directly from the other three conditions. For another ideal $J\in \cJ(P)$, we have \[I\cap J=q_n^{-1}p_n \cdots q_1^{-1} p_1 p_1^{-1} q_1 \cdots p_n^{-1} q_n J.\]

Now consider the semigroup $R\rtimes M$. We start by describing the set of constrictible ideals for $R\rtimes M$. 

\begin{definition} For a subset $N\subset R$ and $p\in M$. Define $p N=\{pn: n\in N\}$ and $p^{-1} N=\{n: pn\in N\}$. Define the set of constructible right ideals of $R$ from $M$ to be: 
\[\cJ_M(R)=\{q_n^{-1}p_n \cdots q_1^{-1} p_1 R: p_1,\cdots, p_n, q_1,\cdots, q_n \in M\}\cup\{\emptyset\}\]

Similar to the case of constructible right ideals of a semigroup, $\cJ_M(R)$ is the smallest set of right $R$-ideals that contains $\emptyset$ and $R$, and for any $I\in\cJ_M(R)$ and $p\in M$, $pI,p^{-1}I\in \cJ_M(R)$. 
\end{definition} 

We would like to point out that elements of $\cJ_M(R)$ are right ideals of the ring $R$, which is closed under right multiplication and addition. In many cases, ideals in $\cJ_M(R)$ have a much simpler form. For example, when $R$ is the ring of algebraic integers in a number field and $M=R^\times$, elements in $\cJ_M(R)$ can always be described in the form of a fractional ideal $q_1^{-1} p_1 R$. This phenomenon is also documented in the case of congruence monoids \cite{Bruce2020}. 

Elements of $\cJ_M(R)$ correspond to constructible right ideals $\cJ(M)$ of the left-cancellative semigroup $M$. For a constructible right ideal $I\in\cJ(M)$, where \[I=q_n^{-1}p_n \cdots q_1^{-1} p_1  M,\]
denote:
\[R_I=q_n^{-1}p_n \cdots q_1^{-1} p_1  R.\]
It is clear from the definition that for any $a\in M$,
\[a R_I=R_{a I}, a^{-1} R_I=R_{a^{-1} I}.\]
As a result, one can proceed almost verbatim as the proof of  \cite[Lemma 3.3]{XLi2012} to show the following Lemma.

\begin{lemma} For any $I,J\in\cJ(M)$, $R_I\cap R_J=R_{I\cap J}$. 
\end{lemma}

\begin{lemma}\label{lm.RM.ideal} Let $I\in \cJ(R)$ and $x\in R$. For $a\in M$, we have
\[a (x+R_I) = ax+R_{aI};\]
and,
\[a^{-1}  (x+R_I) = \begin{cases} r+R_{a^{-1}I} &\mbox{ if there exists }w\in I, x+w=ar, \\
\emptyset &\mbox{ if } (x+R_I)\cap aR=\emptyset.
\end{cases}\]
\end{lemma} 

\begin{proof} It is easy to check that $a (x+R_I) = ax+R_{aI}$. For $a^{-1} (x+R_I)=\{y\in R: ay\in x+R_I\}$: if $(x+R_I)\cap aR=\emptyset$, then it is impossible to find $y\in R$ so that $ay\in (x+R_I)\cap aR$. Therefore, $a^{-1} (x+R_I)=\emptyset$ in this case. Otherwise, there exists $w\in R_I$ and $r\in R$ so that $ar=x+w\in (x+R_I)\cap aR$. We claim that $a^{-1} (x+R_I)=r+a^{-1}\cdot R_I$. 

For any $y\in a^{-1}R_I$, $ay\in R_I$ so that 
\[a(r+y)=ar+ay=x+(w+ay)\in x+R_I.\]
Therefore, $r+a^{-1} R_I\subset a^{-1}\cdot (x+R_I)$. On the other hand, if $az\in x+R_I$, then $az-x=az-ar-w \in R_I$. Since $w\in R_I$, we have $a(z-r)\in R_I$ and thus $z-r\in a^{-1} R_I$ and $z\in r+a^{-1}R_I$. 
\end{proof}

We are now ready to characterize all constrictible right ideals of the semigroup $R\rtimes M$. 

\begin{proposition}\label{prop.RM.ideal} Let $\cJ(R)$ be the set of ideals in $R$. Then the set of constructible right ideals for $R\rtimes M$ can be described as 
\[\cJ(R\rtimes M)=\{(x+R_I)\times I: x\in R, I\in \cJ(M)\}.\]
\end{proposition}

\begin{proof} For any ideal $I\in\cJ(M)$ where 
\[I=q_n^{-1}p_n \cdots q_1^{-1} p_1 M, p_i, q_i\in M,\]
the ideal $(x+R_I)\times I$ can be constructed by
\[(x,1)(0,q_n)^{-1} (0,p_n) \cdots (0,q_1)^{-1} (0,p_1) R\rtimes M.\]

Therefore, the right hand side is a subset of $\cJ(R\rtimes M)$. Because it clearly contains $\emptyset$ and $R\rtimes M$, it suffices to verify that it is closed under forward and backward translations. 

For any $(x+R_I)\times I$ and $(y,a)\in R\rtimes M$, one can easily verify that 
\[(y,a)\cdot (x+R_I)\times I=(y+ax+R_{aI})\times aI.\]
On the other hand,
\[(y,a)^{-1}\cdot (x+R_I)\times I=(0,a)^{-1}(x-y+R_I)\times I.\]
By Lemma \ref{lm.RM.ideal}, this is either $\emptyset$ if $x-y+R_I\cap aR=\emptyset$ or $(r+R_{a^{-1}I})\times a^{-1}I$ if there exists $w\in R_I, r\in R$ with $x-y+w=ar$. 
\end{proof} 

It is often useful to describe the intersection of two constructible ideals of $R\rtimes M$.

\begin{lemma}\label{lm.intersection} Let $I,J\in\cJ(M)$ and $x,y\in R$.

If $w\in (x+R_I)\cap (y+R_J)$, 
\[((x+R_I)\times I) \cap ((y+R_J)\times J) = (w+R_{I\cap J})\times I\cap J.\]

If $(x+R_I)\cap (y+R_J)=\emptyset$, 
\[((x+R_I)\times I) \cap ((y+R_J)\times J)=\emptyset.\]
\end{lemma} 

\begin{proof} It is clear that when $(x+R_I)\cap (y+R_J)=\emptyset$, $((x+R_I)\times I) \cap ((y+R_J)\times J)=\emptyset$. If there exists $w\in (x+R_I)\cap (y+R_J)$, it suffices to prove that $(x+R_I)\cap (y+R_J)=w+R_{I\cap J}$. Let $w=x+s=y+t$ for $s\in R_I$ and $t\in R_J$. For any $z\in (x+R_I)\cap (y+R_J)$,
\[z-w=(z-x)-s=(z-y)-t \in R_I\cap R_J=R_{I\cap J}.\]
Conversely, for any $v\in R_{I\cap J}=R_I\cap R_J$, $w+v-x=s+v\in R_I$ and $w+v-y=t+v\in R_J$, and thus $w+v\in (x+R_I)\cap (y+R_J)$.
\end{proof} 

For two right ideals $M,N$ of $R$, define $M+N=\{x+y: x\in I, y\in J\}$, which is the smallest right ideal containing both $M$ and $N$. 
We always assume that the family $\{R_I: I\in \cJ(M)\}$ is closed under such operation: for any $I,J\in \cJ(M)$, $R_I+R_J=R_K$ for some element $K\in\cJ(M)$ where $I,J\subset K$. This puts a semi-$\vee$ lattice structure on $\cJ(M)$. We are not sure if this property follows immediately from the definition. Nevertheless, all the examples we considered satisfy this property, and this property is assumed throughout the rest of this paper. 

\begin{proposition}\label{prop.RM.universal} The universal semigroup $C^*$-algebra for $R\rtimes M$ is also the universal $C^*$-algebra generated by $\{u^x\}_{x\in R}$, $\{s_a\}_{a\in M}$, and $\{e_I\}_{I\in \cJ(M)}$, such that
\begin{enumerate}
\item The $u^x$ are unitary and $u^x u^y=u^{x+y}$. The $s_a$ are isometries, and $s_as_b=s_{ab}$. Moreover, $s_a u^x = u^{ax} s_a$ for all $x\in R$ and $a\in M$. 
\item $e_I$ are orthogonal projections and $e_Ie_J=e_{I\cap J}$, $e_M=I$. 
\item $s_a e_I s_a^* = e_{aI}$. 
\item For $x\in R_I$, $u^x e_I=e_I u^x$; for $x\notin R_I$, $e_I u^x e_I=0$.  
\end{enumerate}
\end{proposition} 

\begin{proof} Let $\{v_{(x,a)}, E_J\}$ be a set of generators for the semigroup $C^*$-algebra. Define $u^x=v_{(x,1)}$ and $s_a=v_{(0,a)}$. Since $v_{(x,a)}v_{(y,b)}=v_{(x+ay,ab)}$, one can easily verify that $\{u^x, s_a\}$ satisfies condition (1). For each $I\in \cJ(M)$, define $e_I=E_{R_I\times I}$. Here, $R_I\times  I\in\cJ(R\rtimes M)$ by Proposition \ref{prop.RM.ideal}. It is clear that $e_I e_J=e_{I\cap J}$ so that condition (2) is satisfied. For (3), 
\[s_a e_I s_a^* = v_{(0,a)} E_{R_I\times I} v_{(0,a)^*} = E_{R_{aI}\times aI}=e_{aI}.\]
Finally, if $x\in R_I$, $x+R_I=R_I$ and thus 
\[u^x e_I u^{-x}=u^x E_{R_I\times I} u^{-x}= E_{R_I\times I} = e_I.\]
If $x\notin R_I$, $(x+R_I)\cap R_I=\emptyset$. Therefore,
\[e_I u^x e_I u^{-x} = E_{R_I \times I} E_{(x+R_I)\times I}=0.\]
Therefore, $\{u^x\}_{x\in R}$, $\{s_a\}_{a\in M}$, and $\{e_I\}_{I\in \cJ(M)}$ satisfy conditions (1) - (4). 

For the converse, given a family $\{u^x\}_{x\in R}$, $\{s_a\}_{a\in M}$, and $\{e_I\}_{I\in \cJ(M)}$ that satisfies conditions (1) - (4), define $v_{(x,a)}=u^x s_a$ and 
$E_{(x+R_I)\times I}=u^x e_I u^{-x}$. Since $s_a$ is isometric and $u^x$ is unitary, $v_{(x,a)}$ is an isometry. Moreover, for any $(x,a), (y,b)$, 
\[v_{(x,a)}v_{(y,b)}=u^x s_a u^y s_b = u^x u^{ay} s_as_b=u^{x+ay}s_{ab}=v_{(x,a)(y,b)}.\]
Each $E_{(x+R_I)\times I}$ is clearly an orthogonal projection. For any $x,y\in R$ and $I,J\in\cJ(M)$, 
\[E_{(x+R_I)\times I} E_{(y+R_J)\times J}=u^x e_I u^{-x+y} e_J u^{-y}.\]
If $(x+R_I)\cap (y+R_J)=\emptyset$, $-x+y\notin R_I+R_J$. Since we assumed that $R_I+R_J=R_K$ for some $K\in\cJ(M)$ with $I,J\subset K$, we have $e_K u^{-x+y} e_K=0$ by condition (4), and thus \[ e_I u^{-x+y} e_J=e_I e_K u^{-x+y} e_K e_J=0.\]
If $w\in (x+R_I)\cap (y+R_J)$, we have $w=x+s=y+t$ for $s\in R_I$ and $t\in R_J$. Therefore, $-x+y=s-t$ and we have 
\[u^x e_I u^{-x+y} e_J u^{-y}=u^x e_I u^s u^{-t} e_J u^y = u^{x+s} e_Ie_J u^{-y-t}=u^w e_{I\cap J} u^{-w}.\]
Therefore, by Lemma \ref{lm.intersection}, \[E_{(x+R_I)\times I} E_{(y+R_J)\times J}=E_{(w+R_{I\cap J})\times I\cap J}=E_{((x+R_I)\times I) \cap ((y+R_J)\times J)}.\]
Finally, for any ideal $(x+R_I)\cap I\in \cJ(R\rtimes M)$ and $(y,a)\in R\rtimes M$, 
\begin{align*}
v_{(y,a)} E_{(x+R_I)\cap I} v_{(y,a)}^* &= u^y s_a u^x e_I u^{-x} s_a^* u^{-y} \\
&= u^{y+ax} s_a e_I  s_a^* u^{-(y+ax)} \\
&= u^{y+ax} e_{aI} u^{-(y+ax)} \\
&= E_{(y+ax+R_{aI})\times aI} = E_{(y,a) \cdot (x+R_I)\cap I}.
\end{align*}
On the other hand,
\[
v_{(y,a)}^* E_{(x+R_I)\cap I} v_{(y,a)} = s_a^* u^{-y+x} e_I u^{-x+y}  s_a. \]
If there exists $w\in R_I$ and $r\in R$ with $x-y+w=ar$, we have
\begin{align*}
s_a^* u^{-y+x} e_I u^{-x+y}  s_a  &= s_a^* u^{ar} u^{-w} e_I u^w u^{-ar} s_a \\
&= u^r s_a^* u^{-w} u^w e_I s_a u^{-r} \\
&= u^r s_a^* e_I s_a u^{-r} \\
&= u^r e_{a^{-1}I} u^{-r} \\
&= E_{(r+R_{a^{-1}I})\times a^{-1}I} = E_{(y,a)^{-1} \cdot (x+R_I)\cap I}.
\end{align*}
Otherwise, if $x-y+R_I\cap aR=\emptyset$, we have $(y,a)^{-1} \cdot (x+R_I)\cap I=\emptyset$. In this case, $x-y\notin aR+R_I$, where, $aR+R_I=R_K$ for some $K\in\cJ(M)$ with $aM, I\subset K$. Therefore, 
\[e_{aM} u^{x-y} e_I = e_{aM} e_K u^{x-y} e_K e_I =0.\]
By condition (3), $e_{aM}=s_a s_a^*$. Hence, $s_a^* u^{-y+x} e_I=0$ so that \[v_{(y,a)}^* E_{(x+R_I)\cap I} v_{(y,a)}=0.\]
Therefore, $\{v_{(x,a)}, E_{(x+R_I)\times I}\}$ satisfies the conditions for the generators of the semigroup $C^*$-algebra $C^*(R\rtimes M)$. 
\end{proof} 

We now construct a semigroup dynamical system $(\cA_{R,M}, M, \alpha)$, where:
\begin{enumerate}
\item $\cA_{R,M}$ is the $C^*$-subalgebra of $C^*(R\rtimes M)$ generated by $\{u^x, e_I: x\in R, I\in \cJ(M)\}$.
\item $\alpha_a:\cA_{R,M}\to \cA_{R,M}$ defined by 
\[\alpha_a(u^x e_I u^y)=u^{ax} e_{aI} u^{ay}.\]
In particular, $\alpha_a(u^x)=u^{ax} e_{aM}$ and $\alpha_a(e_I)=e_{aI}$. 
\end{enumerate}

\begin{lemma} For any $x\in R$ and $I,J\in \cJ(M)$, 
\[ e_I u^x e_J =\begin{cases} u^{x_1} e_{I\cap J} u^{x_2}, &\mbox{ if } x=x_1+x_2, x_1\in R_I, x_2\in R_J; \\
0, &\mbox{ if } x\notin R_I+R_J. \\
\end{cases}\]
\end{lemma} 

\begin{proof}
Let $I,J\in \cJ(M)$, we assumed $R_I+R_J=R_K$ for some $K\in\cJ(M)$ with $I,J\subset K$. For any $x\in R$: if $x\notin R_I+R_J$, then 
\[e_I u^x e_J = e_I e_{K} u^x e_{K} e_J = 0;\]
if $x\in R_I+R_J$, write $x=x_1+x_2$ with $x_1\in R_I$ and $x_2\in R_J$. We have:
\[e_I u^x e_J = e_I u^{x_1} u^{x_2} e_J = u^{x_1} e_{I\cap J} u^{x_2}.\qedhere\]
\end{proof}

\begin{lemma}\label{lm.ARspan} The $\cA_{R,M}$ can be described as: 
\[\cA_{R,M}=\overline{\lspan}\{u^x e_I u^y: x,y\in R, I\in \cJ(M)\}.\]
\end{lemma} 

\begin{proof} It suffices to prove that the right hand side is an algebra. For any $x,y,w,z\in R$ and $I,J\in\cJ(M)$, consider the product of two elements  $u^x e_I u^y$ and $u^w e_J u^z$:
\[u^x e_I u^y u^w e_J u^z=u^x e_I u^{y+w} e_J u^z.\]
Here, $e_I u^{y+w} e_J=0$ if $y+w\notin R_I+R_J$, and $e_I u^{y+w} e_J=u^c e_{I\cap J} u^d$ if $y+w=c+d$ for some $c\in R_I$ and $d\in R_J$. Therefore, the product is either $0$ or $u^{x+c} e_{I\cap J} u^{d+z}$, and thus the linear span is an algebra.
\end{proof} 

Now, we show that $(\cA_{R,M},M,\alpha)$ is an injective semigroup dynamical system where each $\alpha_a(\cA_{R,M})$ is hereditary and its projections are aligned, so that it fits into the framework of our analysis.

\begin{lemma}\label{lm.alphaa.endo} For each $a\in R^\times$, $\alpha_a$ is an injective $\ast$-endomorphism on $\cA_{R,M}$, and its range $\alpha_a(\cA_{R,M})$ is hereditary, and projections $\{\alpha_a(1)\}$ are commuting. 
\end{lemma} 

\begin{proof} We first prove that $\alpha_a$ is multiplicative: for any $x,y,w,z\in R$ and $I,J\in\cJ(M)$, Lemma \ref{lm.ARspan} proves that 
\[u^x e_I u^y u^wa e_J u^z = \begin{cases} u^{x+c} e_{I\cap J} u^{d+z} &\mbox{ if } y+w=c+d, c\in R_I, d\in R_J;  \\
0 &\mbox{ if } y+w\notin R_I+R_J.
\end{cases}\]
By definition, $\alpha_a(u^x e_I u^y)=u^{ax} e_{aI} u^{ay}$ and $\alpha_a(u^w e_J u^z)=u^{aw} e_{aJ} u^{az}$. We have $y+w\in R_I+R_J$ if and only if $ay+aw\in a(R_{I}+R_{J})=R_{aI}+R_{aJ}$. Moreover, if $y+w=c+d\in R_I+R_J$ with $c\in R_I, d\in R_J$, we also have $ay+aw=ac+ad\in R_{aI}+R_{aJ}$ with $ac\in R_{aI}, ad\in R_{aJ}$. Therefore, one can check \[\alpha_a(u^x e_I u^y u^wa e_J u^z)=\alpha_p(u^x e_I u^y)\alpha_a(u^w e_J u^z).\] 

To see $\alpha_a(\cA_{R,M})$ is hereditary: it is clear that $\alpha_a(\cA_{R,M})\subset \alpha_a(1)\cA\alpha_a(1)$. For the other inclusion, for any generator $u^x e_I u^y\in \cA_R$, consider
\[\alpha_a(1) u^x e_I u^y \alpha_a(1)=e_{aR} u^x e_I u^y e_{aR}.\]
If $x\notin aR+R_I$ or $y\notin R_I+aR$, this becomes $0$. Otherwise, let $x=aw+c$ and $y=d+az$ for some $c,d\in R_I$. We have,
\[e_{aR} u^x e_I u^y e_{aR}=u^{aw} e_{aR\cap I} u^{c+d} e_{aR\cap I} u^{az}.\]
Now, if $c+d\notin aR\cap R_I$, $e_{aR\cap I} u^{c+d} e_{aR\cap I}=0$. Otherwise, $c+d=at$ for some $t\in R$ and $at\in I$. In this case, it becomes
\[u^{a(w+t)} e_{aR\cap I} u^{az} = \alpha_a(u^{w+t} e_{a^{-1}(aR\cap I)} u^{z})\in\alpha_a(\cA_{R,M}).\]
Here, $a^{-1}(aR\cap I)=a^{-1}I$ is also in $\cJ(M)$. Therefore, $\alpha_a(1)\cA_{R,M}\alpha_a(1)=\alpha_a(\cA_{R,M})$, and it is hereditary by Proposition \ref{prop.hereditary}. The projection $\alpha_a(1)=e_{aR}$ which are commuting projections since $e_{aR}e_{bR}=e_{aR\cap bR}$. 

To see $\alpha_a$ is injective, we saw that
\[\alpha_a(\cA_{R,M})=\alpha_a(1)\cA_{R,M}\alpha_a(1)=\lspan\{u^{ax} e_{aI} u^{ay}: x,y\in R, I\in \cJ(M)\}.\]
Therefore, $\alpha_a$ has a left-inverse $\alpha_a^{-1}:\alpha_a(\cA_{R,M})\to\cA_{R,M}$ by mapping $u^{ax} e_{aI} u^{ay}$ to $u^x e_I u^y$, and hence $\alpha_a$ is injective. 
\end{proof} 

For the dynamical system  $(\cA_{R,M},M,\alpha)$, we defined $\cA_{R,M}\rtimes_\alpha^* M$ to be the universal $C^*$-algebra with respect to isometric covariant representations of this semigroup dynamical system. We first show that it coincides with the semigroup $C^*$-algebra for $R\rtimes M$.

\begin{proposition}\label{prop.iso1} The $\cA_{R,M}\rtimes_\alpha^* M$ is isomorphic to $C^*(R\rtimes M)$. 
\end{proposition} 

\begin{proof} By Proposition \ref{prop.RM.universal}, $C^*(R\rtimes M)$ is the universal $C^*$-algebra generated by $\{u^x\}_{x\in R}$, $\{s_a\}_{a\in M}$, and $\{e_I\}_{I\in \cJ(M)}$ such that they satisfy conditions (1) - (4) in Proposition \ref{prop.RM.universal}. We have to show that such generators are in one-to-one correspondence with an isometric covariant representation for $(\cA_{R,M}, M, \alpha)$. Let $\{u^x, s_a, e_I\}$ be the generator for the universal $C^*$-algebra  $C^*(R\rtimes M)$, and recall that $\cA_{R,M}$ is the $C^*$-subalgebra generated by $\{u^x, e_I\}$. 

Given $\{U^x, S_a, E_I\}\subset\bh{H}$ satisfying conditions (1) - (4) in Proposition \ref{prop.RM.universal}, define $\pi:\cA_{R,M}\to\bh{H}$ by $\pi(u^x e_I u^y)=U^x E_I U^y$ and extend linearly. We treat $\{S_a\}_{a\in R^\times}$ as an isometric representation $S:R\to\bh{H}$ with $S(a)=S_a$. We claim that $(\pi,S)$ is an isometric covariant representation of $(\cA_{R,M}, M, \alpha)$. Indeed, $\pi$ is a unital $\ast$-homomorphism and $S$ is an isometric representation. To verify it satisfies the covariance condition, for any $x,y\in R$, $I\in\cJ(M)$, and $a\in M$:
\begin{align*}
S_a \pi(u^x e_I u^y) S_a^* &= S_a U^x E_I U^y S_a^* \\
&= U^{ax} S_a E_I S_a^* U^{ay} \\
&= U^{ax} E_{aI} U^{ay} \\
&= \pi(\alpha_a(u^x e_I u^y)).
\end{align*}
Therefore, for any element $r\in \cA_{R,M}$, $S_a\pi(r) S_a^*=\pi(\alpha_a(r))$. 

Conversely, take an isometric covariant representation $(\pi,S)$. Define $U^x=\pi(u^x)$, $E_I=\pi(e_I)$, and $S_a=S(a)$. We need to show that $\{U^x, S_a, E_I\}$ satisfies conditions (1) - (4) in Proposition \ref{prop.RM.universal}. Since $\pi$ is a unital $\ast$-homomorphism, conditions (2) and (4) are automatically satisfied. For (3),
\[S_a E_I S_a^* = S_a \pi(e_I) S_a^* = \pi(\alpha_a(e_I))=E_{aI}.\]
For (1), $U^x$ are unitary since $\pi$ is a unital $\ast$-homomorphism. Because $S$ is an isometric representation,  $S_a$ are all isometries and $S_aS_b=S_{ab}$. Finally, 
\[S_a U^x S_a^* = S_a \pi(u^x) S_a^* = \pi(\alpha_a(u^x))=\pi(u^{ax} e_{aM})=U^{ax} E_{aM}.\]
By (3), $E_{aM}=S_aS_a^*$, so that $S_a U^x S_a^* = U^{ax} S_a S_a^*$. Multiply $S_a$ on the right gives $S_a U^x = U^{ax} S_a$. 
\end{proof} 

One can define a non-self-adjoint operator algebra associated with the semigroup dynamical system $(\cA_{R,M},M,\alpha)$. In doing so, one has to replace the condition $s_a e_I s_a^*=e_{aI}$ by a non-self-adjoint analogue $V_a E_I=E_{aI} V_a$. It turns out this is precisely the universal non-self-adjoint operator algebra for isometric right covariant representation. 

\begin{proposition}\label{prop.nsa.eq} Let $(\pi,V)$ be an isometric right covariant representation for $(\cA_{R,M},M,\alpha)$. Define $U^x=\pi(u^x)$, $E_I=\pi(e_I)$, and $V_a=V(a)$ for all $x\in R, a\in R^\times$, and $I\in \cJ(M)$. Then,
\begin{enumerate}
\item The $U^x$ are unitary and $U^x U^y=U^{x+y}$. The $V_a$ are isometries and $V_a V_b=V_{ab}$. Moreover, $V_a U^x = U^{ax} V_a$ for all $a\in R^\times$ and $x\in R$.
\item The $E_I$ are projections and $E_I E_J=E_{I\cap J}$, $E_M=I$. 
\item We have $V_a E_I=E_{aI} V_a$.
\item If $x\in I$, $U^x E_I=E_I U^x$; if $x\notin I$, $E_I U^x E_I=0$. 
\end{enumerate}

Conversely, if $\{U^x, E_I, V_a\}$ satisfies conditions (1) - (4), then there exists an isometric right covariant representation $(\pi,V)$ for $(\cA_{R,M},M,\alpha)$ so that $U^x=\pi(u^x)$, $E_I=\pi(e_I)$, and $V_a=V(a)$.
\end{proposition}

\begin{proof} It suffices to verify that $V_a U^x = U^{ax} V_a$ and $V_a E_I=E_{aI} V_a$ because all other conditions are inherited from $\cA_{R,M}$ via the unital $\ast$-homomorphism $\pi$. We have 
\[V_a E_I = V_a \pi(e_I)=\pi(\alpha_a(e_I)) V_a = E_{aI} V_a,\]
and,
\[V_a U^x = V_a \pi(u^x) = \pi(\alpha_a(u^x)) V_a = U^{ax} E_{aM} V_a=U^{ax} V_a E_M=U^{ax} V_a.\]

Conversely, define $\pi(u^x e_I u^y)=U^x E_I U^y$ and $V(a)=V_a$. Then $\pi$ is a unital $\ast$-homomorphism since relations among $\{u^x, e_I\}$ are also satisfied by $\{U^x, E_I\}$. Moreover,
\begin{align*}
V(a) \pi(u^x e_I u^y) &= V_a U^x E_I U^y \\
&= U^{ax} E_{aI} U^{ay} V_a \\
&= \pi(\alpha_a(u^x e_I u^y)) V(a)
\end{align*}
Therefore, $(\pi,V)$ is an isometric right covariant representation. \end{proof} 

As a result of Theorem \ref{thm.envelope}, for a large class of semigroups $R\rtimes M$, their semigroup $C^*$-algebra $C^*(R\rtimes M)$ is the $C^*$-envelope of the universal non-self-adjoint operator algebra with respect to isometric right covariant representation for  $(\cA_{R,M},M,\alpha)$. 

\begin{theorem}\label{thm.envelope.RM} The $C^*$-envelope of $\cA_{R,M}\rtimes_\alpha^{iso} M$ is $C^*(R\rtimes M)$. 
\end{theorem}

We devote the rest of this section to a number of examples. 

\begin{example}\label{ex.CDL} The semigroup $C^*$-algebra $C^*(R\rtimes R^\times)$ (also denoted $\fT_R$) was first studied by Cuntz, Deninger, and Laca \cite{CDL2013}. The original motivation of this paper is that Wiart \cite{Jaspar2016} proved that $C^*(R\rtimes R^\times)$ is the $C^*$-envelope of $\cA_R\rtimes_{\alpha'}^{iso} R^\times$ for some semigroup dynamical system $(\cA_R, R^\times, \alpha')$. Here, $\cA_R$ is the $C^*$-subalgebra of $C^*(R\rtimes R^\times)$ generated by $u^x, e_I$, and $\alpha'_a(u^x)=u^{ax}$, $\alpha'_a(e_I)=e_{aI}$. One may notice the slight difference between the definition $\alpha'$ and our definition (where $\alpha_a(u^x)=u^{ax} e_{aR}$). For an isometric $\ast$-representation $(\pi,S)$, Proposition \ref{prop.iso1} proves that $S_a \pi(u^x) S_a^*=\pi(\alpha_a(u^x))$, but not for $\alpha_a'$. This is the reason we adopted this slightly different action.
Nevertheless, Proposition \ref{prop.nsa.eq} proves that the non-self-adjoint operator algebra constructed by two actions are the same: in \cite{Jaspar2016}, the non-self-adjoint operator algebra $\cA_R\rtimes_{\alpha'}^{iso} R^\times$ is defined as the universal non-self-adjoint operator algebra generated from conditions (1) - (4) of the Proposition \ref{prop.nsa.eq}. Proposition \ref{prop.nsa.eq} now clarifies that this is precisely $\cA_R\rtimes_\alpha^{iso} R^\times$ in our definition. 

It is known that the $ax+b$ semigroup $C^*$-algebra $C^*(R\rtimes R^\times)$ is functorial \cite[Proposition 3.2, Theorem 4.13]{CDL2013}. In other words, let $R$ and $S$ be rings of algebraic integers where $R\subset S$, then there exists an injective $*$-homomorphism from   $C^*(R\rtimes R^\times)$ to $C^*(S\rtimes S^\times)$. By Corollary \ref{cor.functor}, the functoriality of the $C^*$-algebra translates to the functoriality of the non-self-adjoint operator algebra. In other words, there exists a canonical completely isometric homomorphism from  $\cA_R\rtimes_\alpha^{iso} R^\times$ to  $\cA_S\rtimes_\alpha^{iso} S^\times$.
\end{example} 

\begin{example} In many examples, a quotient of the $C^*$-algebra $\cA_{R,M}$ is considered instead. For example, Larsen and Li  \cite{LarsenLi2012} considered a $C^*$-algebra for the semigroup $\mathbb{Z}\rtimes \langle 2\rangle$. In our setup, the ring $R=\mathbb{Z}$ and the unital semigroup $M$ is singly generated by $2$. They considered the universal $C^*$-algebra $\cQ_2$ generated by $u, s_2$, where $u$ is unitary and $s_2$ is an isometry, such that 
\[s_2 u = u^2 s_2,\]
and,
\[s_2 s_2^* + u s_2 s_2^* u^* = I \]

The $C^*$-algebra $\cQ_2$ is in fact a quotient of the semigroup $C^*$-algebra $C^*(\mathbb{Z}\rtimes \langle 2\rangle)$, where the latter only requires $s_2 s_2^* + u s_2 s_2^* u^* \leq I$. 
Denote $e_n=s_2^n s_2^{n*}$. We have $e_1 + ue_1 u^*=I$. 
Let $\cB_2$ be the $C^*$-subalgebra generated by $\{u^x, e_n: n\in\mathbb{N}, x,y\in \mathbb{Z}\}$. Define an $M$-action given by $\alpha_{2^n}(u^x)=u^{2^n x} e_n$ and $\alpha_{2^n}(e_m)=e_{m+n}$. By Proposition \ref{prop.nsa.eq}, the non-self-adjoint operator algebra $\cB_{2}\rtimes_\alpha^{iso} M$ is the universal operator algebra generated by a unitary $U$, orthogonal projections $E_n$, and an isometry $S_2$ such that 
\begin{enumerate}
\item $S_2 U=U^2 S_2$
\item $I=E_0 \geq E_1 \geq E_2 \cdots$
\item $S_2 E_n = E_{n+1} S_2$
\item If $2^n | x$, $U^x E_n=E_n U^x$; if $2^n \nmid x$, $E_n U^x E_n=0$.  
\item $E_1 + U E_1 U^*=I$. 
\end{enumerate}

Here, the extra condition $E_1 + U E_1 U=I$ ensures that $\{U,E_n\}$ corresponds to a $\ast$-representation of $\cB_2$. Theorem \ref{thm.envelope.RM} implies that $C^*_{env}(\cB_{2}\rtimes_\alpha^{iso} M)=\cQ_2$.
\end{example} 

\begin{example} As a generalization to Larsen and Li's construction, Barlack, Omland, and Stameier recently considered a $C^*$-algebra $\cQ_S$ \cite{BOS2018}. Here, $S$ is a family of relatively prime natural numbers, which generate a unital sub-semigroup $M=\langle S\rangle\subset \mathbb{N}^\times$. The $\cQ_S$ is the universal $C^*$-algebra generated by a unitary $u$ and isometries $\{s_p\}_{p\in S}$, such that
\[s_p s_q=s_{pq}, s_p u = u^p s_p, \mbox{ and } \sum_{m=0}^{p-1} u^m s_ps_p^* u^{m*}=I.\]

The isometries $s_p$ corresponds to an isometric representation of the unital semigroup $M$. For each $h\in M$, one can uniquely define $s_h=\prod_{p\in F} s_p$ for $h=\prod_{p\in F} p$. Denote $e_h=s_h s_h^*$. The $C^*$-algebra $\cQ_S$ can be viewed as the universal $C^*$-algebra for the semigroup dynamical system $(\cB_S, M, \alpha)$, where $\cB_S$ is the $C^*$-subalgebra generated by $u$ and $e_{n}$, and $\alpha_a(u^x e_n u^y)=u^{ax} e_{an} u^{ay}$. 

The non-self-adjoint operator algebra $\cB_S\rtimes_\alpha^{iso} M$ is the universal operator algebra generated by a unitary $U$, orthogonal projections $\{E_h\}_{h\in M}$, and isometries $\{S_p\}_{p\in S}$ such that 
\begin{enumerate}
\item $S_p U=U^p S_p$
\item For $h\leq k$, $E_h\geq E_k$. 
\item $S_p E_h = E_{ph} S_p$
\item If $h | x$, $U^x E_h=E_h U^x$; if $h \nmid x$, $E_h U^x E_h=0$.  \item For each $p\in S$, $\sum_{m=0}^{p-1} U^m E_p U^{m*}=I$.  
\end{enumerate}
Theorem \ref{thm.envelope.RM} shows that $C^*_{env}(\cB_S\rtimes_\alpha^{iso} M)=\cQ_S$. 

\end{example}

\section{Additional Examples and Applications}\label{sec.misc}

\counterwithin{theorem}{subsection} 

\subsection{Isometric and Unitary Representation}\label{ssec.trivial}

Consider the injective semigroup dynamical system $(\mathbb{C},P,id)$, where an abelian semigroup $P$ act trivially on $\mathbb{C}$. 
Any unital $\ast$-homomorphism $\pi:\mathbb{C}\to\bh{H}$ must be given by $\pi(\lambda)=\lambda I$. 

For any isometric representation $V$ of $P$, $(\pi, V)$ is an isometric right covariant representation since for any $p\in P$ and $\lambda\in\mathbb{C}$, \[V(p)\pi(\lambda)=\lambda V(p)=\pi(id_p(\lambda)) V(p).\] 
The non-self-adjoint operator algebra $\mathbb{C}\rtimes^{iso}_{id} P$ is the universal operator algebra generated by isometric representations of $P$. 

On the other hand, $(\pi,V)$ is an isometric covariant representation if and only if for each $p\in P$, $V(p)V(p)^*=\pi(id_p(1))=I$. Therefore, in this case, $V(p)$ is in fact unitary. Therefore, the $C^*$-algebra $\mathbb{C}\rtimes^{*}_{id} P$ is the universal $C^*$-algebra generated by unitary representations of $P$. Since $P$ is abelian, we can embed $P$ inside a group $G$ so that $G=P^{-1}P$. Unitary representations of $P$ corresponds to unitary representations of $G$. Therefore, $\mathbb{C}\rtimes^{*}_{id} P\cong C^*(G)$.

Theorem \ref{thm.envelope} gives the following result:

\begin{corollary}\label{cor.trivial} Let $\cT_P^+$ be the universal operator algebra generated by isometric representation of an abelian semigroup $P$, and let $G\supset P$ be a group such that $G=P^{-1}P$. Then \[C^*_{env}(\cT_P^+)=C^*(G).\]
\end{corollary}

This is a known result to experts in the field (for example, it can be inferred by Laca's dilation on Ore semigroups \cite{Laca2000}). 

\begin{example} Let $P=\mathbb{N}$. An isometric representation of $\mathbb{N}$ is generated by a single isometry $V$. In this case, $\cT_\mathbb{N}^+$ is the disk algebra $\cA(\mathbb{D})$, and it is a well-known fact that its $C^*$-envelope is $C^*(\mathbb{Z})\cong C(\mathbb{T})$.  
\end{example} 

\subsection{The Cuntz Algebra $\cO_k$ and UHF core} 

The celebrated Cuntz algebra $\cO_k$, first studied in  \cite{Cuntz1977}, is the universal $C^*$-algebra generated by $k$ isometries $W_i$ with \[\sum_{i=1}^k W_i W_i^* = I.\]
For a word $\mu=a_1a_2\cdots a_n\in\mathbb{F}_k^+$ where each $a_i\in\{1,\cdots, k\}$, we say the length of $\mu$ is $n$, denoted by $|\mu|=n$. We also denote $W_\mu=W_{a_1}W_{a_2}\cdots W_{a_n}$. The UHF core is defined to be the $C^*$-subalgebra $\cA_k$ generated by $\{W_\mu W_\nu^*: |\mu|=|\nu|\}$. For any word $\mu,\nu$, it is well known that
\[W_\mu^* W_\nu = \begin{cases} W_{\nu'} & \mbox{ if } \nu=\mu\nu'; \\
W_{\mu'}^* & \mbox{ if } \mu=\nu\mu'; \\
0 & \mbox{ otherwise.} 
\end{cases}\]
Therefore, one can check that 
\[\cA_k=\overline{\lspan}\{ W_\mu W_\nu^*: |\mu|=|\nu|\}.\]
Moreover, for each $n$, let
\[\cA_k^n=\overline{\lspan}\{ W_\mu W_\nu^*: |\mu|=|\nu|=n\},\]
and it is isomorphic to the matrix algebra $\cM_{k^n}$, where elements of the form $W_\mu W_\nu^*$ act as matrix units. Consequently, $\cA_k$ is a UHF-algebra of type $k^\infty$ (see \cite[Section V.4]{KenCStarByExample}). 

Define $\alpha_1:\cA_k\to\cA_k$ be the $\ast$-endomorphism where $\alpha_1(W_\mu W_\nu^*)=W_{1\mu} W_{1\nu}^*$, which is the conjugation by $W_1$. This is an injective $\ast$-endomorphism on $\cA_k$, and one can check that its image is 
\[\alpha_1(\cA_k)=\overline{\lspan}\{ W_1 W_{\mu} W_{\nu}^* W_1^*: |\mu|=|\nu|\}.\]
As a result, $\alpha_1(\cA_k)=\alpha_1(1)\cA_k\alpha_1(1)$ and thus $\alpha_1(\cA_k)$ is hereditary. 

One can associate a semigroup dynamical system $(\cA_k, \mathbb{N}, \alpha)$, where $\alpha_n=\alpha_1^n$.
For any $n,m\in\mathbb{N}$, \[\alpha_n(1)\alpha_m(1)=W_1^n W_1^{n*} W_1^m W_1^{m*}=W_1^{\max\{n,m\}} W_1^{\max\{n,m\}*}.\]
Moreover, 
\[\alpha_n(\cA_k)=\overline{\lspan}\{ W_1^k W_{\mu} W_{\nu}^* W_1^{k*}: |\mu|=|\nu|\},\]
which is also equal to $\alpha_n(1)\cA_k\alpha_n(1)$, and thus is hereditary. Therefore, this injective semigroup dynamical system satisfies that each $\alpha_n(\cA_k)$ is hereditary and its projections are aligned, thereby fitting inside the framework considered in this paper. 

\begin{proposition} The Cuntz algebra $\cO_k$ is isomorphic to the $C^*$-algebra $\cA_k\rtimes_\alpha^* \mathbb{N}$. 
\end{proposition} 

\begin{proof} Let $(\pi_S,S)$ be an isometric covariant representation of the dynamical system $(\cA_k, \mathbb{N}, \alpha)$. $S$ is uniquely determined by a single isometry $S_1=S(1)$ and $S(n)=S_1^n$. For each $1\leq i\leq k$, $S_i=\pi_S(W_i W_1^*) S(1)$. Here, for $i=1$, this is well defined since
\[\pi_S(W_1 W_1^*) S(1)=\pi(\alpha_1(I)) S(1)=S(1)\pi_S(I)=S(1).\]
Then, 
\begin{align*}
\sum_{i=1}^k S_i S_i^* &= \sum_{i=1}^k \pi_S(W_i W_1^*) S(1)S(1)^*\pi_S(W_1 W_i^*)  \\
&= \sum_{i=1}^k  \pi_S(W_i W_1^*) \pi_S(W_1W_1^*) \pi_S(W_1 W_i^*) \\
&= \sum_{i=1}^k  \pi_S(W_i W_i^*) \\
&=   \pi_S(\sum_{i=1}^k W_i W_i^*) = \pi_S(I)=I 
\end{align*}
Therefore, $\{S_i\}$ defines a representation of the Cuntz algebra $\cO_k$. Conversely, every family $\{S_i\}$ that generates $\cO_k$ defines an isometric covariant representation $(\pi_S,S)$ by $\pi_S(W_\mu W_\nu^*)=S_\mu S_\nu^*$ and $S(n)=S_1^n$. Therefore, Cuntz isometries is in one-to-one correspondence with isometric covariant representations of  $(\cA_k, \mathbb{N}, \alpha)$, and thus their universal $C^*$-algebras are isomorphic. 
\end{proof}

We can construct a non-self-adjoint operator algebra $\cA \rtimes_\alpha^{iso} \mathbb{N}$, universal with respect to isometric right covariant representation. Theorem \ref{thm.envelope} applies:

\begin{theorem} The Cuntz algebra $\cO_k$ is the $C^*$-envelope of $\cA \rtimes_\alpha^{iso} \mathbb{N}$.
\end{theorem}

\subsection{Semigroup $C^*$-algebras}

Let $P$ be an abelian left-cancellative semigroup, and $\cJ(P)$ be the set of all constructible ideals of $P$. We briefly reviewed Xin Li's construction of its semigroup $C^*$-algebra $C^*(P)$ in Definition \ref{def.semigroupC}. In particular, in Section \ref{sec.ex}, we focused on the case of $R\rtimes M$. Here, we explore another semigroup dynamical system arising from a general left-cancellative semigroup.

Let $\{v_x\}_{x\in P}$ and $\{e_I\}_{I\in\cJ(P)}$ be the generators for $C^*(P)$. Let $\cD_P$ be the $C^*$-subalgebra generated by $\{e_I\}_{I\in\cJ(P)}$. Since the product of $e_I$ and $e_J$ is $e_{I\cap J}$, it is clear that 
\[\cD_p=\overline{\lspan}\{e_I: I\in\cJ(P)\}.\]
Define a semigroup dynamical system $(\cD_P, \alpha,P)$ by $\alpha_p(e_I)=e_{pI}$. One can easily verify that the universal $C^*$-algebra $\cD_p\rtimes_\alpha^* P$ with respect to isometric covariant representations is precisely the semigroup $C^*$-algebra $C^*(P)$ in the sense of Xin Li. 

On the other hand, the non-self-adjoint operator algebra $\cD_p\rtimes_\alpha^{iso} P$ can be described as the universal operator algebra generated by a similar set of generators.

\begin{proposition} The non-self-adjoint operator algebra $\cD_p\rtimes_\alpha^{iso} P$ is the universal operator algebra generated by isometries $\{V_p\}_{p\in P}$ and orthogonal projections $\{E_I\}_{I\in\cJ(P)}$, such that 
\begin{enumerate}
\item $V_pV_q=V_{pq}$,
\item $V_p E_I = E_{pI} V_p$,
\item $E_\emptyset=0$ and $E_P=I$,
\item $E_I E_J=E_{I\cap J}$. 
\end{enumerate}
\end{proposition} 

\begin{proof} For each isometric right covariant representation $(\pi,V)$ of $(\cD_P,\alpha,P)$, define $E_I=\pi(e_I)$ and $V_p=V(p)$. It is simple to verify that $\{V_p\}_{p\in P}$ and $\{E_I\}_{I\in\cJ(P)}$ satisfies conditions (1) - (4). 

Conversely, given such $\{V_p\}_{p\in P}$ and $\{E_I\}_{I\in\cJ(P)}$, define $V(p)=V_p$ and $\pi(e_I)=E_I$. Condition (1) ensures that $V$ is an isometric representation; conditions (3) and (4) ensure that $\pi$ is a unital $*$-homomorphism of $\cD_P$; condition (2) ensures that $(\pi,V)$ is an isometric right covariant representation. 

Now a generatoring family $\{V_p\}_{p\in P}$ and $\{E_I\}_{I\in\cJ(P)}$ is in one-to-one correspondence with an isometric right covariant representation, and thus their universal operator algebras coincide. 
\end{proof} 

As an application of Theorem \ref{thm.envelope},

\begin{corollary}\label{cor.cstarP} The semigroup $C^*$-algebra $C^*(P)$ is the $C^*$-envelope of $\cD_p\rtimes_\alpha^{iso} P$. 
\end{corollary} 

\begin{remark} There is a number of examples where the $C^*$-envelope of a non-self-adjoint operator algebra associated with a semigroup is the boundary quotient of the semigroup $C^*$-algebra, a smaller $C^*$-algebra compared to $C^*(P)$. For example, the $C^*$-envelope of the disk algebra $\cA(\mathbb{D})$ is $C(\mathbb{T})$, a quotient of the semigroup $C^*$-algebra $C^*(\mathbb{N})$; the $C^*$-envelope of the non-commutative disk algebra $\fA_k$ is the Cuntz algebra $\cO_k$, a quotient of the semigroup $C^*$-algebra $C^*(\mathbb{F}_k^+)$. However, in our setup, the $C^*$-algebra $\cA$ in the semigroup dynamical system is rigid: it is preserved faithfully in the $C^*$-envelope. For example, in the semigroup dynamical system $(\cD_P, P, \alpha)$, the diagonal $C^*$-algebra $\cD_P$ inside the non-self-adjoint operator algebra $\cD_P\rtimes_\alpha^{iso} P$ is preserved when we take the $C^*$-envelope, and therefore, we obtain the full semigroup $C^*$-algebra $C^*(P)$ instead of its boundary quotient. 
\end{remark} 

\begin{example} Consider the unital semigroup $P=\{1+4n: n\in\mathbb{N}\}$. This is a particular example of a congruence monoid, whose constructible ideals are characterized by Bruce  \cite{Bruce2020}: $\cJ(P)=\{I_a: 2\nmid a\}\cup\{\emptyset\}$ where $I_a=a\mathbb{N} \cap P$. For odd numbers $a$ and $b$, $I_a \cap I_b=I_{\lcm(a,b)}$. Therefore, the semigroup $C^*$-algebra $C^*(P)$ is the universal $C^*$-algebra generated by isometries $\{v_p\}_{p\in P}$ and projections $\{e_a\}_{2\nmid a}$ such that 
\begin{enumerate}
\item $v_pv_q=v_{pq}$,
\item $v_p e_a v_p^* = e_{pa}$,
\item $e_\emptyset=0$ and $e_1=I$,
\item $e_a e_b=e_{\lcm(a,b)}$. 
\end{enumerate}

Now the non-self-adjoint operator algebra $\cD_P\rtimes_\alpha^{iso} P$ is the universal operator algebra generated by isometries $\{V_p\}_{p\in P}$ and projections $\{E_a\}_{2\nmid a}$ such that
\begin{enumerate}
\item $V_pVv_q=V_{pq}$,
\item $V_p E_a  = E_{pa} V_p$,
\item $E_\emptyset=0$ and $E_1=I$,
\item $E_a E_b=E_{\lcm(a,b)}$. 
\end{enumerate}

Corollary \ref{cor.cstarP} implies that $C^*_{env}(\cD_P\rtimes_\alpha^{iso} P)=C^*(P)$. 
\end{example} 

\bibliographystyle{abbrv}
\bibliography{semigroup}

\def\lfhook#1{\setbox0=\hbox{#1}{\ooalign{\hidewidth
  \lower1.5ex\hbox{'}\hidewidth\crcr\unhbox0}}}
\begin{thebibliography}{10}

\bibitem{Arveson2011}
W.~Arveson.
\newblock The noncommutative {C}hoquet boundary {II}: hyperrigidity.
\newblock {\em Israel J. Math.}, 184:349--385, 2011.

\bibitem{ArvesonSubalgI}
W.~B. Arveson.
\newblock Subalgebras of {$C\sp{\ast} $}-algebras.
\newblock {\em Acta Math.}, 123:141--224, 1969.

\bibitem{BOS2018}
S.~Barlak, T.~Omland, and N.~Stammeier.
\newblock On the {$K$}-theory of {$C^{\ast}$}-algebras arising from integral
  dynamics.
\newblock {\em Ergodic Theory Dynam. Systems}, 38(3):832--862, 2018.

\bibitem{Bruce2020}
C.~Bruce.
\newblock {$\rm C^*$}-algebras from actions of congruence monoids on rings of
  algebraic integers.
\newblock {\em Trans. Amer. Math. Soc.}, 373(1):699--726, 2020.

\bibitem{Cuntz1977}
J.~Cuntz.
\newblock Simple {$C\sp*$}-algebras generated by isometries.
\newblock {\em Comm. Math. Phys.}, 57(2):173--185, 1977.

\bibitem{Cuntz2000}
J.~Cuntz.
\newblock {$C^*$}-algebras associated with the {$ax+b$}-semigroup over {$\Bbb
  N$}.
\newblock In {\em {$K$}-theory and noncommutative geometry}, EMS Ser. Congr.
  Rep., pages 201--215. Eur. Math. Soc., Z\"{u}rich, 2008.

\bibitem{CDL2013}
J.~Cuntz, C.~Deninger, and M.~Laca.
\newblock {$C^*$}-algebras of {T}oeplitz type associated with algebraic number
  fields.
\newblock {\em Math. Ann.}, 355(4):1383--1423, 2013.

\bibitem{CuntzLi2010}
J.~Cuntz and X.~Li.
\newblock The regular {$C^\ast$}-algebra of an integral domain.
\newblock In {\em Quanta of maths}, volume~11 of {\em Clay Math. Proc.}, pages
  149--170. Amer. Math. Soc., Providence, RI, 2010.

\bibitem{KenCStarByExample}
K.~R. Davidson.
\newblock {\em {$C^*$}-algebras by example}, volume~6 of {\em Fields Institute
  Monographs}.
\newblock American Mathematical Society, Providence, RI, 1996.

\bibitem{DFK2014}
K.~R. Davidson, A.~H. Fuller, and E.~T.~A. Kakariadis.
\newblock Semicrossed products of operator algebras by semigroups.
\newblock {\em Mem. Amer. Math. Soc.}, 247(1168):v+97, 2017.

\bibitem{DK2011}
K.~R. Davidson and E.~G. Katsoulis.
\newblock Operator algebras for multivariable dynamics.
\newblock {\em Mem. Amer. Math. Soc.}, 209(982):viii+53, 2011.

\bibitem{DK2012}
K.~R. Davidson and E.~G. Katsoulis.
\newblock Semicrossed products of the disk algebra.
\newblock {\em Proc. Amer. Math. Soc.}, 140(10):3479--3484, 2012.

\bibitem{DK2015}
K.~R. Davidson and M.~Kennedy.
\newblock The {C}hoquet boundary of an operator system.
\newblock {\em Duke Math. J.}, 164(15):2989--3004, 2015.

\bibitem{AdamElias2020}
A.~Dor-On and E.~Katsoulis.
\newblock Tensor algebras of product systems and their {${\rm C}^*$}-envelopes.
\newblock {\em J. Funct. Anal.}, 278(7):108416, 32, 2020.

\bibitem{DM2005}
M.~A. Dritschel and S.~A. McCullough.
\newblock Boundary representations for families of representations of operator
  algebras and spaces.
\newblock {\em J. Operator Theory}, 53(1):159--167, 2005.

\bibitem{DP2015}
B.~L. Duncan and J.~R. Peters.
\newblock Operator algebras and representations from commuting semigroup
  actions.
\newblock {\em J. Operator Theory}, 74(1):23--43, 2015.

\bibitem{Fuller2013}
A.~H. Fuller.
\newblock Nonself-adjoint semicrossed products by abelian semigroups.
\newblock {\em Canad. J. Math.}, 65(4):768--782, 2013.

\bibitem{Hamana1979}
M.~Hamana.
\newblock Injective envelopes of operator systems.
\newblock {\em Publ. Res. Inst. Math. Sci.}, 15(3):773--785, 1979.

\bibitem{Humeniuk2020}
A.~Humeniuk.
\newblock {C}*-envelopes of semicrossed products by lattice ordered abelian
  semigroups.
\newblock {\em https://arxiv.org/abs/2001.07294}, 2020.

\bibitem{Kakariadis2016}
E.~T.~A. Kakariadis.
\newblock Semicrossed products of {$C^*$}-algebras and their {$C^*$}-envelopes.
\newblock {\em J. Anal. Math.}, 129:1--31, 2016.

\bibitem{EvgElias2012}
E.~T.~A. Kakariadis and E.~G. Katsoulis.
\newblock Semicrossed products of operator algebras and their {${\rm
  C}^*$}-envelopes.
\newblock {\em J. Funct. Anal.}, 262(7):3108--3124, 2012.

\bibitem{KP2013}
E.~T.~A. Kakariadis and J.~R. Peters.
\newblock Representations of {$\rm C^*$}-dynamical systems implemented by
  {C}untz families.
\newblock {\em M\"{u}nster J. Math.}, 6(2):383--411, 2013.

\bibitem{KR2019}
E.~G. Katsoulis and C.~Ramsey.
\newblock Crossed products of operator algebras.
\newblock {\em Mem. Amer. Math. Soc.}, 258(1240):vii+85, 2019.

\bibitem{Laca2000}
M.~Laca.
\newblock From endomorphisms to automorphisms and back: dilations and full
  corners.
\newblock {\em J. London Math. Soc. (2)}, 61(3):893--904, 2000.

\bibitem{LacaRaeburn1996}
M.~Laca and I.~Raeburn.
\newblock Semigroup crossed products and the {T}oeplitz algebras of nonabelian
  groups.
\newblock {\em J. Funct. Anal.}, 139(2):415--440, 1996.

\bibitem{Larsen2010}
N.~S. Larsen.
\newblock Crossed products by abelian semigroups via transfer operators.
\newblock {\em Ergodic Theory Dynam. Systems}, 30(4):1147--1164, 2010.

\bibitem{LarsenLi2012}
N.~S. Larsen and X.~Li.
\newblock The 2-adic ring {$C^\ast$}-algebra of the integers and its
  representations.
\newblock {\em J. Funct. Anal.}, 262(4):1392--1426, 2012.

\bibitem{XLi2010}
X.~Li.
\newblock Ring {$C^*$}-algebras.
\newblock {\em Math. Ann.}, 348(4):859--898, 2010.

\bibitem{XLi2012}
X.~Li.
\newblock Semigroup {${\rm C}^*$}-algebras and amenability of semigroups.
\newblock {\em J. Funct. Anal.}, 262(10):4302--4340, 2012.

\bibitem{MS1998}
P.~S. Muhly and B.~Solel.
\newblock An algebraic characterization of boundary representations.
\newblock In {\em Nonselfadjoint operator algebras, operator theory, and
  related topics}, volume 104 of {\em Oper. Theory Adv. Appl.}, pages 189--196.
  Birkh\"{a}user, Basel, 1998.

\bibitem{Murphy1996}
G.~J. Murphy.
\newblock Crossed products of {$C^*$}-algebras by endomorphisms.
\newblock {\em Integral Equations Operator Theory}, 24(3):298--319, 1996.

\bibitem{Nica1992}
A.~Nica.
\newblock {$C\sp *$}-algebras generated by isometries and {W}iener-{H}opf
  operators.
\newblock {\em J. Operator Theory}, 27(1):17--52, 1992.

\bibitem{Peters1984}
J.~Peters.
\newblock Semicrossed products of {$C^\ast$}-algebras.
\newblock {\em J. Funct. Anal.}, 59(3):498--534, 1984.

\bibitem{Sarason1966}
D.~Sarason.
\newblock Invariant subspaces and unstarred operator algebras.
\newblock {\em Pacific J. Math.}, 17:511--517, 1966.

\bibitem{Jaspar2016}
J.~Wiart.
\newblock Dilating covariant representations of a semigroup dynamical system
  arising from number theory.
\newblock {\em Integral Equations Operator Theory}, 84(2):217--233, 2016.

\end{thebibliography}

\end{document}